\documentclass[11pt]{article}

\usepackage{amsmath,amsthm,amssymb,a4wide,hyperref,graphicx,stmaryrd}
\usepackage[all]{xy}
\usepackage{times}

\newcommand{\M}{\mathcal{M}}

\renewcommand{\phi}{\varphi}
\renewcommand{\circledS}{\boxplus}

\newcommand{\BP}{\textbf{P}}

\newcommand{\lr}[1]{\langle #1 \rangle}

\renewcommand{\iff}{\Longleftrightarrow}
\newcommand{\lra}{\leftrightarrow}

\newcommand{\weg}[1]{}

\theoremstyle{definition}
\newtheorem{theorem}{Theorem}%这些在一些package 里都有定义,用的时候看系统提示.
\newtheorem{lemma}[theorem]{Lemma}
\newtheorem{definition}[theorem]{Definition}

\newtheorem{remark}[theorem]{Remark}
\newtheorem{proposition}[theorem]{Proposition}

\newtheorem{example}[theorem]{Example}
\newtheorem{corollary}[theorem]{Corollary}

\title{Two variants of noncontingency operator}
\author{Jie Fan\\
\small School of Humanities, University of Chinese Academy of Sciences, Beijing, China  \\
\small \texttt{jiefan@ucas.ac.cn}}
\date{}

\begin{document}
\maketitle

\begin{abstract}
By slightly adapting two equivalent semantics of noncontingency operator, we obtain two variants, $\boxdot$ and $\boxplus$, with non-equivalent semantics. We show that on the class of models satisfying any of five basic properties (i.e. seriality, reflexivity, transitivity, symmetry, Euclidicity), the logic $\mathcal{L}(\boxdot)$, which has $\boxdot$ as the sole modal primitive, is less expressive than the logic $\mathcal{L}(\boxplus)$, which has $\boxplus$ as the sole modal primitive. We investigate the frame definability of both languages. We then axiomatize $\mathcal{L}(\boxplus)$ and $\mathcal{L}(\boxdot)$ over various classes of bimodal frames. Among other results, a notion of morphisms, called `$\boxdot$-morphisms', are provided to show the completeness of axiomatizations of $\mathcal{L}(\boxdot)$ over serial frames and also over symmetric frames.
\end{abstract}

\section{Introduction}

Past decades have witnessed a bunch of studies on noncontingency logic, see e.g.~\cite{Humberstone95,DBLP:journals/ndjfl/Kuhn95,DBLP:journals/ndjfl/Zolin99,hoeketal:2004,steinsvold:2008,Fanetal:2014,FanvD:neighborhood,Fanetal:2015,Fan:2018b,Fan2018newperspective,Fan:2019}.  This logic is obtained by enriching propositional logic with an important metaphysical notion --- contingency, which dates back to Aristotle~\cite{Borgan67}. Intuitively, a proposition is {\em contingent}, if it is possibly true and also possibly false; otherwise, it is noncontingent, i.e. necessarily true or necessarily false. In an epistemic setting, contingency amounts to `ignorance', and noncontingency amounts to `knowing whether', which is perhaps the closest knowing-wh companion to `knowing that' (namely, standard propositional knowledge) among various knowledge types~\cite{Wang:2016}.

Formally, given a Kripke model $\M=\lr{S,R,V}$, where $S$ is a nonempty set of possible worlds, $R\subseteq S\times S$ is called {\em accessibility relation}, and $V$ is a valuation that assigns a set $V(p)\subseteq S$ to each propositional variable $p$, the formula $\Delta\phi$, read ``it is noncontingent that $\phi$'', is evaluated as follows:

\[
\begin{array}{lclr}
\M,s\vDash\Delta\phi&\iff & \text{for all }t,u\in S,\text{ if }sRt\text{ and }sRu,&\\
&&\text{then }(\M,t\vDash\phi\iff \M,u\vDash\phi).&(\text{DEF~1})\\
\end{array}
\]

Equivalently,
$$
\begin{array}{lclr}
\M,s\vDash\Delta\phi&\iff & R(s)\vDash\phi\text{ or }R(s)\vDash\neg\phi,&(\text{DEF~2})\\
\end{array}$$
where $R(s)\vDash\phi$ means that $\phi$ is true at all successors of $s$ w.r.t. $R$, and similarly for $R(s)\vDash\neg\phi$.
%where for any set $X$ and any formula $\phi$, $X\vDash\phi$ means that $\phi$ is true at any world in $X$, and for any binary relation $T$, $T(s)$ is defined as the set of successors of $s$ w.r.t. $T$; in symbol, $\{t\in S\mid sTt\}$.

%Then it may be natural to consider two variants of the semantics of $\Delta$.
By slightly adapting the above semantics, we obtain two variants of $\Delta$, denoted $\boxdot$ and $\boxplus$ respectively, as follows.
\[
\begin{array}{lclr}
\M,s\vDash\boxdot\phi&\iff &\text{for all }t,u,\text{ if }sR_1t\text{ and }sR_2u,&\\
&&\text{then }(\M,t\vDash\phi\iff \M,u\vDash\phi).&(\text{DEF~1'})\\
\M,s\vDash\boxplus\phi&\iff &R_1(s)\vDash\phi\text{ or }R_2(s)\vDash\neg\phi.&(\text{DEF~2'})\\
\end{array}
\]

It is not hard to see that (DEF~1) and (DEF~2) are, respectively, special cases of (DEF~1') and (DEF~2') when $R_1=R_2=R$. This entails that both $\boxdot$ and $\boxplus$ are more general than $\Delta$.  Moreover, as $\vDash\boxdot\phi\lra\boxdot\neg\phi$ but $\nvDash\boxplus\phi\lra\boxplus\neg\phi$ (as we will see below), we may call $\boxdot$ `general noncontingency' and $\boxplus$ `pseudo noncontingency' operators. Unlike the fact that (DEF~1) is equivalent to (DEF~2), (DEF~1') and (DEF~2') are {\em not} equivalent, that is, $\nvDash\boxdot\phi\lra\boxplus\phi$.

This paper investigates both operators. Roughly speaking, a proposition is {\em generalized noncontingent}, if the proposition has the same truth value no matter whether you look at it in this way ($R_1$) or in that way ($R_2$); and a proposition is {\em pseudo noncontingent}, if it is necessary in this way ($R_1$), or it is impossible in that way ($R_2$). Whenever both ways are the same, both operators then become the more-familiar noncontingency operator.

The remainder of the paper is structured as follows. After introducing the syntax and semantics of logic $\mathcal{L}(\boxdot)$ for generalized noncontingency and logic $\mathcal{L}(\boxplus)$ for pseudo noncontingency (Sec.~\ref{sec.synseman}), we compare the relative expressivity of the two logics (Sec.~\ref{sec.expressivity}), and investigate their frame definability (Sec.~\ref{sec.framedefinability}) with the help of a notion of $\boxdot$-morphisms (Sec.~\ref{sec.morphism}). We then axiomatize $\mathcal{L}(\boxplus)$ and $\mathcal{L}(\boxdot)$ over various bimodal frames in Sec.~\ref{sec.boxplus-logics} and Sec.~\ref{sec.boxdot-logics}, where the completeness of $\mathcal{L}(\boxdot)$ over serial frames and also over symmetric frames are proved via the notion of $\boxdot$-morphisms. We conclude with a few future work in Sec.~\ref{sec.conclusion}.

%which is the closest friend of `knowing that' (viz. standard propositional knowledge) among %notions of `knowing X'.

%This paper proposes two non-contingency style operators: generalized non-contingency and pseudo non-contingency operators. Intuitively, to say some proposition is {\em generalized contingent}, if the proposition has the same truth value no matter whether you look at it in this way or that way; and some proposition is {\em pseudo non-contingency}, if it is necessary in this way, or it is impossible in that way. Whenever these two ways are the same, then both operators become equal to each other and to a non-contingency operator.

%We strictly distinguish $\Box_1$ and $\Box_2$, without resorting to some kind of symmetry.

%nomic necessity, also called `physical necessity' in~\cite{Bacon:1981}, or `causal necessity' in~\cite{Burks:1951}, or natural necessity in~\cite{BigelowPargetter:1990}. It is disputable that the accessibility relation for nomic necessity is reflexive.

%Purely nomic necessity.

\section{Syntax and semantics}\label{sec.synseman}

%The languages involved in the current paper are list below.
Let $\BP$ be a fixed nonempty set of propositional variables.
\begin{definition}[Syntax] Where $p\in\BP$, the language $\mathcal{L}(\boxdot)$ of generalized noncontingency logic and the language $\mathcal{L}(\boxplus)$ of pseudo noncontingency logic are defined inductively as follows.
\[
\begin{array}{ll}
%\mathcal{L}(\Box_1):&\phi::=p\mid \neg\phi\mid (\phi\land\phi)\mid \Box_1\phi\\
%\mathcal{L}(\Box_2):&\phi::=p\mid \neg\phi\mid (\phi\land\phi)\mid \Box_2\phi\\
%\mathcal{L}(\Box_1,\Box_2):&\phi::=p\mid \neg\phi\mid (\phi\land\phi)\mid \Box_1\phi\mid \Box_2\phi\\
%\mathcal{L}(\boxtimes):&\phi::=p\mid \neg\phi\mid (\phi\land\phi)\mid \boxtimes\phi\\
\mathcal{L}(\boxdot):&\phi::=p\mid \neg\phi\mid (\phi\land\phi)\mid \boxdot\phi\\
\mathcal{L}(\boxplus):&\phi::=p\mid \neg\phi\mid (\phi\land\phi)\mid \boxplus\phi\\
%\mathcal{L}()
%\mathcal{L}(\boxtimes,\boxplus):&\phi::=p\mid \neg\phi\mid (\phi\land\phi)\mid \boxtimes\phi\mid \boxplus\phi\\
\end{array}
\]
\end{definition}

$\boxdot\phi$ and $\boxplus\phi$ are read ``it is generalized noncontingent that $\phi$'' and ``it is pseudo noncontingent that $\phi$'', respectively. As we will see below, the comparisons between the two languages are interesting, both in expressivity and in axiomatizations.% especially in axiomatizations, since the minimal axiomatizations of them are very close, except that one has a crucial axiom whereas the other lacks.

The languages are interpreted on bimodal models. To say that $\mathcal{M}=\lr{S,R_1,R_2,V}$ is a {\em bimodal model}, if $S$ is a nonempty set of possible worlds, $R_1$ and $R_2$ are accessibility relations over $S$, and $V$ is a function assigning to each propositional variable a subset of $S$. A bimodal frame $\mathcal{F}$ is a bimodal model without valuations. If $R_1$ and $R_2$ both possess a property $P$ (such as seriality, reflexivity, transitivity, symmetry, Euclidicity), then $\M$ ($\mathcal{F}$) is called a $P$ bimodal model (resp. a $P$ bimodal frame). Moreover, $R_i(s)=\{t\in S\mid sR_it\}$ for $i=1,2$.

\begin{definition}[Semantics] Given a bimodal model $\M=\lr{S,R_1,R_2,V}$ and $s\in S$, the semantics of both languages is defined as follows.
\[
\begin{array}{|lll|}
\hline
\M,s\vDash p&\iff & s\in V(p)\\
\M,s\vDash\neg\phi&\iff &\M,s\nvDash\phi\\
\M,s\vDash\phi\land\psi&\iff &\M,s\vDash\phi\text{ and }\M,s\vDash\psi\\
%\M,s\vDash\Box_i\phi&\iff&R_i(s)\vDash\phi,\text{ where }i\in\{1,2\}\\
%\M,s\vDash\Box_1\phi&\iff&R_1(s)\vDash\phi\\
%\M,s\vDash\Box_2\phi&\iff&R_2(s)\vDash\neg\phi\\
%\M,s\vDash\boxtimes\phi&\iff&R_1(s)\vDash\phi \text{ and }R_2(s)\vDash\neg\phi\\
\M,s\vDash\boxdot\phi&\iff &\text{for all }t,u,\text{ if }sR_1t\text{ and }sR_2u,\\
&&\text{then }(\M,t\vDash\phi\iff \M,u\vDash\phi)\\
\M,s\vDash\boxplus\phi&\iff &R_1(s)\vDash\phi\text{ or }R_2(s)\vDash\neg\phi\\
\hline
\end{array}
\]
\end{definition}
Where for $i\in\{1,2\}$, $R_i(s)\vDash\phi$ stands for ``for all $t\in R_i(s), \M,t\vDash\phi$'', and $R_i(s)\nvDash\phi$ for the negation of this claim, that is, ``for some $t\in R_i(s)$, $\M,t\nvDash\phi$''.  Obviously, when $R_1(s)=\emptyset$ or $R_2(s)=\emptyset$, it holds vacuously that $\M,s\vDash\boxplus\phi$ and $\M,s\vDash\boxdot\phi$ for all $\phi$; if $R_1=R_2$, then $\boxplus=\boxdot$ and each of them becomes an operator for noncontingency. %If $R_i(s)=\emptyset$ for $i\in\{1,2\}$, then $\M,s\vDash\boxplus\phi$ for all $\phi$.

It is noteworthy remarking that $\vDash\boxdot\phi\lra(\boxplus\phi\land\boxplus\neg\phi)$, as can be seen more clearly from an alternative semantical definition for $\boxdot$.%the fact that the above definition for $\boxdot$ is equivalent to the following:
\[
\begin{array}{lll}
\M,s\vDash\boxdot\phi&\iff &(R_1(s)\vDash\phi\text{ or }R_2(s)\vDash\neg\phi)\text{ and}\\
&&(R_1(s)\vDash\neg\phi\text{ or }R_2(s)\vDash\phi).\\
\end{array}
\]
Consequently, $\boxplus$ is deductively weaker than $\boxdot$. In contrast, as Sec.~\ref{sec.expressivity} will show, $\boxplus$ is deductively stronger than $\boxdot$, equivalently, $\boxdot$ is expressively weaker than $\boxplus$.\footnote{As for the definitions of `deductively weaker' and `expressively weaker', we refer to~\cite{Fan:2017}.}

Note that $\nvDash\boxplus\phi\lra\boxplus\neg\phi$ but $\vDash\boxdot\phi\lra\boxdot\neg\phi$. To see the former, consider a model $\M=\lr{S,R_1,R_2,V}$ in which $S=\{s,t,u\}$, $R_1(s)=\{t\}$ and $R_2(s)=\{u\}$, and $V(p)=\{t\}$. Then it should be easily verified that $s\vDash\circledS p$ but $s\nvDash\circledS\neg p$. This will matter when we look into the differences between axiomatizations of $\circledS$-logics and of $\boxdot$-logics.% Moreover, $\vDash\boxdot\phi\lra(\boxplus\phi\land\boxplus\neg\phi)$.

We may define $\M,s\vDash\Box_i\phi$ as $R_i(s)\vDash\phi$, where $i\in\{1,2\}$, then $\boxplus\phi$ is equivalent to $\Box_1\phi\vee\Box_2\neg\phi$. The operator $\boxplus$, written $N'''$ on ~\cite[p.~229]{Humberstone:2016philosophicalapplications}, to our knowledge, has not been axiomatized in the literature. %As noticed on~\cite[p.~229]{Humberstone:2016philosophicalapplications}, if $\Box_2=\Box_1$, then $\boxplus$, which is written $N'''$ therein instead, is an operator for noncontingency. It may be worth noting that the axiomatization for $\boxplus$ is lacking in the literature.

If we read $\Box_i\phi$ as ``the agent $i$ believes that $\phi$'', then it is not hard to see that the negation of $\boxdot$ characterizes the notion of weak belief-disagreement in~\cite{ChenPan:2018}: one agent fails to believe one proposition and the other fails to believe its negation. In that paper, the notion is mentioned in passing only, which is based on serial bimodal frames. %As we will see, however, the complete proof for the serial logic is not trivial.

On serial bimodal frames, the semantics of $\boxdot$ is equivalent to
\[
\begin{array}{lll}
\M,s\vDash\boxdot\phi&\iff &(R_1(s)\vDash\phi\text{ and }R_2(s)\vDash\phi)\text{ or }(R_1(s)\vDash\neg\phi\text{ and }R_2(s)\vDash\neg\phi).\\
\end{array}
\]

The epistemic meaning of this definition is that agents 1 and 2 have the same knowledge about $\phi$, i.e. they both know $\phi$, or they both know $\neg\phi$; in a doxastic reading, it means `agents 1 and 2 have the belief agreement on $\phi$'.
%The epistemic meaning of this definition is agents 1 and 2 have the same knowledge about $\phi$, i.e. they both know $\phi$, or they both know $\neg\phi$. In a doxastic reading, $\boxdot\phi$ means `agents 1 and 2 have the belief agreement on $\phi$'.

%If we read $\Box_i\phi$ as ``the agent $i$ believes that $\phi$'', then it is not hard to see that the negation of $\boxdot$ characterizes the notion of weak belief-disagreement in~\cite{JiaPan:2017strong}. As we will see, however, the complete proof of the serial logic is not trivial.

\weg{\begin{proposition}
For any $\phi\in\mathcal{L}(\boxdot)$, for any class $\mathbb{C}$ of bimodal frames, there exists a class of monomodal frames $\mathbb{F}$ such that, if $\mathbb{C}\nvDash\phi$ then $\mathbb{F}\nvDash\phi^{\Delta}$.

For any $\phi\in\mathcal{L}(\Delta)$, for any class $\mathbb{F}$ of monomodal frames, there exists a class of bimodal frames $\mathbb{C}$ such that, if $\mathbb{F}\nvDash\phi$ then $\mathbb{C}\nvDash\phi^{\boxdot}$.
\end{proposition}}

To simplify the proofs later, we claim the following results, which should be easily verified.

\begin{proposition}\label{prop.boxdotto}
Let $i,j\in\{1,2\}$ and $i\neq j$ and $R_j(s)\neq \emptyset$. If $\M,s\vDash\boxdot\phi$, then $\M,s\vDash\Delta_i\phi$.
\end{proposition}

Note that the converse fails. For example, in a model $\M$, $s$ has only a single $R_i$-successor $t$, and another single $R_j$-successor $u$, whereas $t$ and $u$ have different truth values for $\phi$. In spite of this, the converse indeed holds when $R_i(s)$ and $R_j(s)$ has a common element.

\begin{proposition}\label{prop.boxdotfrom}
Suppose that $R_1(s)\cap R_2(s)\neq \emptyset$. If $\M,s\vDash\Delta_1\phi\land\Delta_2\phi$, then $\M,s\vDash\boxdot\phi$.
\end{proposition}

\begin{corollary}\label{coro.equiv}
Suppose that $R_1(s)\cap R_2(s)\neq \emptyset$. Then $\M,s\vDash\Delta_1\phi\land\Delta_2\phi$ iff $\M,s\vDash\boxdot\phi$.
\end{corollary}

\weg{\begin{proposition}\label{prop.boxdotfrom}
Let $i,j\in\{1,2\}$ and $i\neq j$ and $R_i(s)\cap R_j(s)\neq \emptyset$. If $\M,s\vDash\Delta_i\phi\land\Delta_j\phi$, then $\M,s\vDash\boxdot\phi$.
\end{proposition}}

%This will make our proofs later much simpler.  %since it may be the case that $s$ has two $R_i$ successors that agree on $\phi$, but a $R_j$ successor that does not agree on $\phi$ with both $R_i$ successors.
\weg{\begin{corollary}\label{coro.equiv}
Let $i,j\in\{1,2\}$ and $i\neq j$ and $R_i(s)\cap R_j(s)\neq \emptyset$. Then $\M,s\vDash\boxdot\phi$ iff $\M,s\vDash\Delta_i\phi\land\Delta_j\phi$.
\end{corollary}}

\section{Expressivity}\label{sec.expressivity}

This section compares the relative expressivity of $\mathcal{L}(\boxdot)$ and $\mathcal{L}(\boxplus)$. It turns out that the former is less expressive than the latter on all five classes of basic bimodal models.

To make our presentation self-contained, we introduce some necessary technical terms.
\begin{definition} Let $\mathcal{L}_1$ and $\mathcal{L}_2$ be two languages that are interpreted on the same class of models $\mathbb{C}$.
\begin{itemize}
\item $\mathcal{L}_2$ is at least as expressive as $\mathcal{L}_1$, notation: $\mathcal{L}_1\preceq \mathcal{L}_2$, if for all $\phi\in\mathcal{L}_1$, there exists $\psi\in\mathcal{L}_2$ such that for all $\M$ in $\mathbb{C}$ and all $s$ in $\M$, we have that $\M,s\vDash\phi$ iff $\M,s\vDash\psi$.
\item $\mathcal{L}_1$ and $\mathcal{L}_2$ are equally expressive, notation: $\mathcal{L}_1\equiv \mathcal{L}_2$, if $\mathcal{L}_1\preceq \mathcal{L}_2$ and $\mathcal{L}_2\preceq \mathcal{L}_1$.
\item $\mathcal{L}_1$ is less expressive than $\mathcal{L}_2$, notation: $\mathcal{L}_1\prec \mathcal{L}_2$, if $\mathcal{L}_1\preceq \mathcal{L}_2$ but $\mathcal{L}_2\not\preceq \mathcal{L}_1$.
\end{itemize}
\end{definition}

\begin{proposition}\label{prop.less-expressive-k}
$\mathcal{L}(\boxdot)$ is less expressive than $\mathcal{L}(\boxplus)$ on the class of all bimodal models, the class of serial bimodal models, the class of transitive bimodal models, the class of Euclidean bimodal models.
\end{proposition}

\begin{proof}
We have already seen that $\boxdot$ is definable in $\mathcal{L}(\boxplus)$, as $\vDash\boxdot\phi\lra\boxplus\phi\land\boxplus\neg\phi$. This entails that $\mathcal{L}(\boxdot)\preceq\mathcal{L}(\boxplus)$.
%This entails that $\mathcal{L}(\boxplus)$ is at least as expressive as $\mathcal{L}(\boxdot)$.

%To show the converse does not hold, consider the following serial, transitive, Euclidean bimodal models:
To show $\mathcal{L}(\boxplus)\not\preceq\mathcal{L}(\boxdot)$, consider the following serial, transitive, Euclidean bimodal models:
\weg{Consider two serial bimodal models $\M=\lr{S,R_1,R_2,V}$ and $\M'=\lr{S',R_1',R_2',V'}$, where $S=\{s,t,u\}$, $R_1=\{(s,t),(t,t),(u,u)\}$, $R_2=\{(s,u),(t,t),(u,u)\}$, $V(p)=\{s,t\}$ and $V(q)=S$ for all $q\neq p$, and $S'=\{s',t',u'\}$, $R_1'=\{(s',t'),(t',t'),(u',u')\}$, $R_2'=\{(s',u'),(t',t'),(u',u')\}$, $V'(p)=\{s',u'\}$ and $V(q)=S'$ for all $q\neq p$. This can be pictured as follows, where we omit the valuation of propositional variables other than $p$.}
\[
\xymatrix{\M&t:p\ar@(ul,ur)|{1,2}&s:p\ar[l]|1\ar[r]|2&u:\neg p\ar@(ur,ul)|{1,2}&&\M'&t':\neg p\ar@(ul,ur)|{1,2}&s':p\ar[l]|1\ar[r]|2&u':p\ar@(ur,ul)|{1,2}}
\]
One can check that $\M,s\vDash\boxplus p$ and $\M',s'\nvDash\boxplus p$, thus the $\mathcal{L}(\boxplus)$-formula $\boxplus p$ can distinguish $(\M,s)$ and $(\M',s')$.

However, $(\M,s)$ and $(\M',s')$ cannot be distinguished by any $\mathcal{L}(\boxdot)$-formula. That is, for all $\phi\in\mathcal{L}(\boxdot)$, we have $(\M,s\vDash\phi\iff \M',s'\vDash\phi)$. The proof proceeds by induction on $\phi$.

The base case and Boolean cases are straightforward. For the case $\boxdot\phi$, we have
\[\begin{array}{ll}
&\M,s\vDash\boxdot\phi\\
\iff &(\M,t\vDash\phi\iff\M,u\vDash\phi)\\
\stackrel{(\ast)}\iff &(\M',u'\vDash\phi\iff\M',t'\vDash\phi)\\
\iff &\M',s'\vDash\boxdot\phi,\\
\end{array}\]
where $(\ast)$ holds since $\M,t\vDash\phi$ iff $\M',u'\vDash\phi$, and $\M,u\vDash\phi$ iff $\M',t'\vDash\phi$ for all $\phi\in\mathcal{L}(\boxdot)$, as can be easily verified.
%To show this, we first show $(\ast)$: for all $\phi\in\mathcal{L}(\boxdot)$, we have $(\M,t\vDash\phi\iff \M',u'\vDash\phi)$ and $(\M,t'\vDash\phi\iff \M',u\vDash\phi)$.
%The proof proceeds with induction on $\phi$. The nontrivial case is $\boxdot\phi$, for which one may easily check that $t\vDash\boxdot\phi$ and $t'\vDash\boxdot\phi$ and $u\vDash\boxdot\phi$ and $u'\vDash\boxdot\phi$.
%We have thus shown $(\ast)$.
%We
\weg{$$(\M,t\vDash\phi\iff \M,u\vDash\phi)\iff (\M',t'\vDash\phi\iff \M',u'\vDash\phi).$$
If $\phi\lra p$ is globally true on both $\M$ and $\M'$, then by the construction of $\M$ and $\M'$, $(\M,t\vDash \phi\not\iff\M,u\vDash\phi)$ and $(\M',t'\vDash\phi\not\iff \M',u'\vDash\phi)$, and thus the statement holds. We need only consider $\phi$ such that $\phi\lra p$ is {\em not} globally true on both $\M$ and $\M'$. For this, we show $(\ast)$: $t\vDash\phi\iff t'\vDash\phi$ and $u\vDash\phi$ iff $u'\vDash\phi$ by induction on $\phi$.

$\phi=q\in\BP$, thus $q\neq p$. It is obvious that $(\ast)$ holds.

Boolean cases follow directly from the semantics and induction hypothesis.

$\phi=\boxdot\psi$. One may easily check that $t\vDash\boxdot\psi$ and $t'\vDash\boxdot\psi$ and $u\vDash\boxdot\psi$ and $u'\vDash\boxdot\psi$.}
%One has only $R_1$-successors, the other has only $R_2$-successors.
\end{proof}

%\begin{corollary}
%$\mathcal{L}(\boxdot)$ is less expressive than $\mathcal{L}(\boxplus)$ on the class of all bimodal models.
%\end{corollary}

\begin{proposition}\label{prop.less-expressive-b}
$\mathcal{L}(\boxdot)$ is less expressive than $\mathcal{L}(\boxplus)$ on the class of symmetric bimodal models.
\end{proposition}

\begin{proof}
Again, $\mathcal{L}(\boxdot)\preceq\mathcal{L}(\boxplus)$. For the strict part,
consider the following symmetric bimodal models:
\[
\xymatrix{\M&t:p\ar[r]|1&s:p\ar[l]|1\ar[r]|2&u:\neg p\ar[l]|2&&\M'&t':\neg p\ar[r]|1&s':p\ar[l]|1\ar[r]|2&u':p\ar[l]|2}
\]

First, $\M,s\vDash\boxplus p$ but $\M',s'\nvDash\boxplus p$. This means that $(\M,s)$ and $(\M',s')$ can be distinguished by $\mathcal{L}(\boxplus)$.

Second, as shown in Prop.~\ref{prop.less-expressive-k}, we can prove that for all $\phi\in\mathcal{L}(\boxdot)$, $\M,s\vDash\phi$ iff $\M',s'\vDash\phi$. Then $(\M,s)$ and $(\M',s')$ cannot be distinguished by $\mathcal{L}(\boxdot)$.
%One has only $R_1$-successors, the other has only $R_2$-successors.
\end{proof}

%check! how to define the product of frames: $K\times K$

\begin{proposition}\label{prop.less-expressive-t}
$\mathcal{L}(\boxdot)$ is less expressive than $\mathcal{L}(\boxplus)$ on the class of reflexive bimodal models.
\end{proposition}

\begin{proof}
%We have already seen that $\boxdot$ is definable in $\mathcal{L}(\boxplus)$, as $\vDash\boxdot\phi\lra\boxplus\phi\land\boxplus\neg\phi$.
Again, $\mathcal{L}(\boxdot)\preceq\mathcal{L}(\boxplus)$.

For the strict part, consider the following reflexive bimodal models:
\weg{Consider two serial bimodal models $\M=\lr{S,R_1,R_2,V}$ and $\M'=\lr{S',R_1',R_2',V'}$, where $S=\{s,t,u\}$, $R_1=\{(s,t),(t,t),(u,u)\}$, $R_2=\{(s,u),(t,t),(u,u)\}$, $V(p)=\{s,t\}$ and $V(q)=S$ for all $q\neq p$, and $S'=\{s',t',u'\}$, $R_1'=\{(s',t'),(t',t'),(u',u')\}$, $R_2'=\{(s',u'),(t',t'),(u',u')\}$, $V'(p)=\{s',u'\}$ and $V(q)=S'$ for all $q\neq p$. This can be pictured as follows, where we omit the valuation of propositional variables other than $p$.}
\[
\xymatrix{\M&t:p\ar@(ul,ur)|{1,2}&s:p\ar[l]|1\ar[r]|2\ar@(ur,ul)|{1,2}&u:\neg p\ar@(ur,ul)|{1,2}&&\M'&t':\neg p\ar@(ul,ur)|{1,2}&s':p\ar[l]|1\ar[r]|2\ar@(ur,ul)|{1,2}&u':p\ar@(ur,ul)|{1,2}}
\]
First, $\M,s\vDash\boxplus p$ but $\M',s'\nvDash\boxplus p$, thus $\boxplus p$ can distinguish $(\M,s)$ and $(\M',s')$.

However, no $\mathcal{L}(\boxdot)$-formula can distinguish both pointed models. That is, for all $\phi\in\mathcal{L}(\boxdot)$, we have $(\M,s\vDash\phi\iff \M',s'\vDash\phi)$. The proof proceeds with induction on $\phi$. We only consider the nontrivial case $\boxdot\phi$.
%The base case and Boolean cases are straightforward. For the case $\boxdot\phi$, we have
\[\begin{array}{ll}
&\M,s\vDash\boxdot\phi\\
\iff &(\M,s\vDash \phi\iff \M,t\vDash\phi)\text{ and }(\M,s\vDash\phi\iff\M,u\vDash\phi)\text{ and }(\M,t\vDash\phi\iff\M,u\vDash\phi)\\
\stackrel{(\ast)}\iff &(\M',s'\vDash \phi\iff \M',u'\vDash\phi)\text{ and }(\M,s'\vDash\phi\iff\M,t'\vDash\phi)\text{ and }(\M',u'\vDash\phi\iff\M',t'\vDash\phi)\\
\iff &\M',s'\vDash\boxdot\phi,\\
\end{array}\]
where $(\ast)$ is the case due to the induction hypothesis that $\M,s\vDash\phi\iff \M',s'\vDash\phi$, and the fact that $\M,t\vDash\phi$ iff $\M',u'\vDash\phi$, and $\M,u\vDash\phi$ iff $\M',t'\vDash\phi$ for all $\phi\in\mathcal{L}(\boxdot)$, as can be easily verified.
%To show this, we first show $(\ast)$: for all $\phi\in\mathcal{L}(\boxdot)$, we have $(\M,t\vDash\phi\iff \M',u'\vDash\phi)$ and $(\M,t'\vDash\phi\iff \M',u\vDash\phi)$.
%The proof proceeds with induction on $\phi$. The nontrivial case is $\boxdot\phi$, for which one may easily check that $t\vDash\boxdot\phi$ and $t'\vDash\boxdot\phi$ and $u\vDash\boxdot\phi$ and $u'\vDash\boxdot\phi$.
%We have thus shown $(\ast)$.
%We
\weg{$$(\M,t\vDash\phi\iff \M,u\vDash\phi)\iff (\M',t'\vDash\phi\iff \M',u'\vDash\phi).$$
If $\phi\lra p$ is globally true on both $\M$ and $\M'$, then by the construction of $\M$ and $\M'$, $(\M,t\vDash \phi\not\iff\M,u\vDash\phi)$ and $(\M',t'\vDash\phi\not\iff \M',u'\vDash\phi)$, and thus the statement holds. We need only consider $\phi$ such that $\phi\lra p$ is {\em not} globally true on both $\M$ and $\M'$. For this, we show $(\ast)$: $t\vDash\phi\iff t'\vDash\phi$ and $u\vDash\phi$ iff $u'\vDash\phi$ by induction on $\phi$.

$\phi=q\in\BP$, thus $q\neq p$. It is obvious that $(\ast)$ holds.

Boolean cases follow directly from the semantics and induction hypothesis.

$\phi=\boxdot\psi$. One may easily check that $t\vDash\boxdot\psi$ and $t'\vDash\boxdot\psi$ and $u\vDash\boxdot\psi$ and $u'\vDash\boxdot\psi$.}
%One has only $R_1$-successors, the other has only $R_2$-successors.
\end{proof}

\begin{remark}
Note that in the proof of Prop.~\ref{prop.less-expressive-t}, $\M$ and $\M'$ are both serial and transitive, but {\em not} Euclidean; for instance, $sR_1t$ and $sR_1s$ but {\em not} $tR_1s$, thus Prop.~\ref{prop.less-expressive-k} cannot be shown by using the constructed models in Prop.~\ref{prop.less-expressive-t}.

The clear-sighted reader may ask whether the Euclidean closures of $\M$ and $\M'$ in Prop.~\ref{prop.less-expressive-t} can handle Prop.~\ref{prop.less-expressive-k} (and even Prop.~\ref{prop.less-expressive-b}) uniformly. That is, if we construct models $\M$ and $\M'$ as follows:
\[
\xymatrix{\M&t:p\ar@(ul,ur)|{1,2}\ar[r]|1&s:p\ar[l]|1\ar[r]|2\ar@(ur,ul)|{1,2}&u:\neg p\ar@(ur,ul)|{1,2}\ar[l]|2&&\M'&t':\neg p\ar@(ul,ur)|{1,2}\ar[r]|1&s':p\ar[l]|1\ar[r]|2\ar@(ur,ul)|{1,2}&u':p\ar@(ur,ul)|{1,2}\ar[l]|2}
\]
then does $\M,s\vDash\phi\iff \M',s'\vDash\phi$ hold for all $\phi\in\mathcal{L}(\boxdot)$?

The answer seems negative. The reason is as follows: to show the case $\boxdot\phi$, that is, $\M,s\vDash\boxdot\phi\iff \M',s'\vDash\boxdot\phi$, (as before) we need to prove that $\M,t\vDash\phi\iff \M',u'\vDash\phi$ for all $\phi\in\mathcal{L}(\boxdot)$ (and also $\M,u\vDash\phi\iff \M',t'\vDash\phi$ for all $\phi\in\mathcal{L}(\boxdot)$), whose case $\boxdot\phi$ relies on showing again that $\M,s\vDash\phi\iff \M',s'\vDash\phi$ for all $\phi\in\mathcal{L}(\boxdot)$. This is a vicious circle.

In comparison, this situation does not occur in the proofs of Prop.~\ref{prop.less-expressive-k}-Prop.~\ref{prop.less-expressive-t}; instead, as any point $x\in \{t,u,t',u'\}$ in the proofs of those propositions has no two different successors with respect to $R_1$ and $R_2$, all formulas of the form $\boxdot\phi$ are true at $x$.
\end{remark}

\section{$\boxdot$-morphisms}\label{sec.morphism}

In this section, we introduce a notion of $\boxdot$-morphisms, which is useful in the proof of frame undefinability and the completeness proof of $\mathcal{L}(\boxdot)$ over serial frames and also over symmetric frames below.

\begin{definition}\label{def.boxdot-morphism}%[$\boxdot$-morphisms]
Let $\M=\lr{S,R_1,R_2,V}$ and $\M'=\lr{S',R_1',R_2',V'}$ be two bimodal models. A function $f:S\to S'$ is a {\em $\boxdot$-morphism} from $\M$ to $\M'$, if for all $x\in S$,
\begin{itemize}
\item[(Var)] For all $p\in\BP$, $x\in V(p)$ iff $f(x)\in V'(p)$,
\item[(Forth)] For any $y,z\in S$, if $xR_1y$ and $xR_2z$ and $f(y)\neq f(z)$, then $f(x)R_1'f(y)$ and $f(x)R_2'f(z)$,
\item[(Back)] For all $y',z'\in S'$, if $f(x)R_1'y'$ and $f(x)R_2'z'$ and $y'\neq z'$, then there are $y,z\in S$ such that $xR_1y$ and $xR_2z$ and $f(y)=y'$ and $f(z)=z'$.
\end{itemize}
We say that $\M'$ is a $\boxdot$-morphic image of $\M$, if there is a surjective $\boxdot$-morphism from $\M$ to $\M'$.
\end{definition}

The following result indicates that $\mathcal{L}(\boxdot)$-formulas and $\mathcal{L}(\boxplus)$-formulas are invariant under $\boxdot$-morphisms.

\begin{proposition}\label{prop.boxdot-modelproperty}
Let $\M=\lr{S,R_1,R_2,V}$ and $\M'=\lr{S',R_1',R_2',V'}$ be two bimodal models, and let $f$ be a $\boxdot$-morphism from $\M$ to $\M'$. Then for all $x\in S$, for all $\phi\in\mathcal{L}(\boxdot)\cup\mathcal{L}(\boxplus)$, we have
$$\M,x\vDash\phi\iff\M',f(x)\vDash\phi.$$
\end{proposition}

\begin{proof}
By induction on $\phi\in\mathcal{L}(\boxdot)\cup\mathcal{L}(\boxplus)$. We only consider the nontrivial case $\boxdot\phi$ and $\boxplus\phi$.

Suppose that $\M,x\nvDash\boxdot\phi$, to show that $\M',f(x)\nvDash\boxdot\phi$. By supposition, there are $y,z\in S$ such that $xR_1y$ and $xR_2z$ and it is {\em not} the case that $(\M,y\vDash\phi\iff\M,z\vDash\phi)$. By induction hypothesis, it is {\em not} the case that $(\M',f(y)\vDash\phi\iff\M',f(z)\vDash\phi)$, which implies that $f(y)\neq f(z)$. Now using (Forth), we obtain $f(x)R'_1f(y)$ and $f(x)R'_2f(z)$. Therefore, $\M',f(x)\nvDash\boxdot\phi$.

Conversely, assume that $\M',f(x)\nvDash\boxdot\phi$, to prove that $\M,x\nvDash\boxdot\phi$. By assumption, there exist $y',z'\in S'$ such that $f(x)R_1'y'$ and $f(x)R_2'z'$ and it is {\em not} the case that $(\M',y'\vDash\phi\iff \M',z'\vDash\phi)$. It is clear that $y'\neq z'$. Using (Back), we infer that there are $y,z\in S$ such that $xR_1y$ and $xR_2z$ and $f(y)=y'$ and $f(z)=z'$, and thus it is {\em not} the case that $(\M',f(y)\vDash\phi\iff \M',f(z)\vDash\phi)$. By induction hypothesis, it is {\em not} the case that $(\M,y\vDash\phi\iff \M,z\vDash\phi)$. Therefore, $\M,x\nvDash\boxdot\phi$.

Suppose that $\M,x\nvDash\boxplus\phi$, to prove that $\M',f(x)\nvDash\boxplus\phi$. By supposition, there exists $y\in S$ such that $xR_1y$ and $\M,y\nvDash\phi$, and there exists $z\in S$ such that $xR_2z$ and $\M,z\nvDash\neg\phi$ (viz. $\M,z\vDash\phi$). By induction hypothesis, $\M',f(y)\nvDash\phi$ and $\M',f(z)\vDash\phi$, which implies that $f(y)\neq f(z)$. Then applying (Forth), we infer that $f(x)R_1'f(y)$ and $f(x)R_2'f(z)$. Therefore, $\M',f(x)\nvDash\boxplus\phi$.

Conversely, assume that $\M',f(x)\nvDash\boxplus\phi$, to demonstrate that $\M,x\nvDash\boxdot\phi$. By assumption, there is a $y'\in S'$ such that $f(x)R_1'y'$ and $\M',y'\nvDash\phi$, and there is a $z'\in S'$ such that $f(x)R_2'z'$ and $\M',z'\nvDash\neg\phi$ (namely, $\M',z'\vDash\phi$). Then $y'\neq z'$. Applying (Back), we derive that there exist $y,z\in S$ such that $xR_1y$ and $xR_2z$ and $f(y)=y'$ and $f(z)=z'$. Thus $\M',f(y)\nvDash\phi$ and $\M',f(z)\vDash\phi$. By induction hypothesis, $\M,y\nvDash\phi$ and $\M,z\vDash\phi$, and therefore $\M,x\nvDash\boxplus\phi$, as desired.
\end{proof}

\section{Frame definability}\label{sec.framedefinability}

This section investigates the frame definability of logics $\mathcal{L}(\boxdot)$ and $\mathcal{L}(\boxplus)$. It turns out that all five basic frame properties, i.e. seriality, reflexivity, transitivity, symmetry, Euclidicity, are not definable in both logics. For this, we adopt the notion of $\boxdot$-morphisms on the frame level, which is obtained from Def.~\ref{def.boxdot-morphism} by leaving out the valuations.

\begin{definition}
Let $\mathcal{F}=\lr{S,R_1,R_2}$ and $\mathcal{F}'=\lr{S',R_1',R_2'}$ be two bimodal frames. A function $f:S\to S'$ is a {\em $\boxdot$-morphism} from $\mathcal{F}$ to $\mathcal{F}'$, if for all $x\in S$,
\begin{itemize}
%\item[(Var)] For all $p\in\BP$, $x\in V(p)$ iff $f(x)\in V'(p)$,
\item[(Forth)] For any $y,z\in S$, if $xR_1y$ and $xR_2z$ and $f(y)\neq f(z)$, then $f(x)R_1'f(y)$ and $f(x)R_2'f(z)$,
\item[(Back)] For all $y',z'\in S'$, if $f(x)R_1'y'$ and $f(x)R_2'z'$ and $y'\neq z'$, then there are $y,z\in S$ such that $xR_1y$ and $xR_2z$ and $f(y)=y'$ and $f(z)=z'$.
\end{itemize}
We say that $\mathcal{F}'$ is a $\boxdot$-morphic image of $\mathcal{F}$, if there is a surjective $\boxdot$-morphism from $\mathcal{F}$ to $\mathcal{F}'$.
\end{definition}

\begin{proposition}\label{prop.boxdot-frameproperty}
Let $\mathcal{F}=\lr{S,R_1,R_2}$ and $\mathcal{F}'=\lr{S',R_1',R_2'}$ be two bimodal frames. If $\mathcal{F}'$ is a $\boxdot$-morphic image of $\mathcal{F}$, then for all $\phi\in\mathcal{L}(\boxdot)\cup\mathcal{L}(\boxplus)$, we have
$$\mathcal{F}\vDash\phi\iff \mathcal{F}'\vDash\phi.$$
\end{proposition}

\begin{proof} Assume that $\mathcal{F}'$ is a $\boxdot$-morphic image of $\mathcal{F}$. Then there is a surjective $\boxdot$-morphism from $\mathcal{F}$ to $\mathcal{F}'$, say $f$.

Suppose that $\mathcal{F}\nvDash\phi$, to show that $\mathcal{F}'\vDash\phi$. By supposition, there exists a valuation $V$ on $\mathcal{F}$ and $s\in S$ such that $\lr{\mathcal{F},V},s\nvDash\phi$. Define a valuation $V'$ on $\mathcal{F}'$ by $V'(p)=\{f(x)\mid x\in V(p)\}$ for all $p\in\BP$. Then $f$ is a $\boxdot$-morphism from $\lr{\mathcal{F},V}$ to $\lr{\mathcal{F}',V'}$. By Prop.~\ref{prop.boxdot-modelproperty} and the fact that $\lr{\mathcal{F},V},s\nvDash\phi$, we obtain $\lr{\mathcal{F}',V'},f(s)\nvDash\phi$, and therefore $\mathcal{F}'\nvDash\phi$.

Conversely, suppose that $\mathcal{F}'\nvDash\phi$, to show that $\mathcal{F}\nvDash\phi$. By supposition, there is a valuation $V'$ on $\mathcal{F}'$ and $s'\in S'$ such that $\lr{\mathcal{F}',V'},s'\nvDash\phi$. Since $f$ is surjective, there must be an $s\in S$ such that $s'=f(s)$. Define a valuation $V$ on $\mathcal{F}$ by $V(p)=\{x\mid f(x)\in V'(p)\}$ for all $p\in\BP$. Then $f$ is a $\boxdot$-morphism from $\lr{\mathcal{F},V}$ to $\lr{\mathcal{F}',V'}$. By Prop.~\ref{prop.boxdot-modelproperty} again and the fact that $\lr{\mathcal{F}',V'},f(s)\nvDash\phi$, we infer that $\lr{\mathcal{F},V},s\nvDash\phi$, and therefore $\mathcal{F}\nvDash\phi$, as desired.
\end{proof}

\begin{proposition}\label{prop.undefinable}
None of seriality, reflexivity, transitivity, symmetry and Euclidicity are definable in $\mathcal{L}(\boxdot)\cup\mathcal{L}(\boxplus)$.
\end{proposition}

\begin{proof}
Consider the following bimodal frames:
\[
\xymatrix{\mathcal{F}:&s\ar[rr]|{1,2}&&t\ar[rr]|{1,2}&&u&&\mathcal{F}':&s'\ar@(ur,ul)|{1,2}}
\]

Define a function $g$ from $\mathcal{F}$ to $\mathcal{F}'$ as follows: $g(s)=g(t)=g(u)=s'$. It is not hard to check that $g$ is a surjective $\boxdot$-morphism, thus $\mathcal{F}'$ is a $\boxdot$-morphic image of $\mathcal{F}$. By Prop.~\ref{prop.boxdot-frameproperty}, $\mathcal{F}\vDash\phi$ iff $\mathcal{F}'\vDash\phi$ for all $\phi\in\mathcal{L}(\boxdot)\cup\mathcal{L}(\boxplus)$.

If seriality were defined by a set of $\mathcal{L}(\boxdot)$-formulas or a set of $\mathcal{L}(\boxplus)$-formulas, say $\Gamma$, then as $\mathcal{F}'$ is serial, $\mathcal{F}'\vDash\Gamma$, and thus $\mathcal{F}\vDash\Gamma$, which would imply that $\mathcal{F}$ should be serial: a contradiction. Thus seriality is not definable in $\mathcal{L}(\boxdot)$. The proofs for the undefinability of other frame properties are analogous.
\end{proof}

The frame undefinability results can be understood in the following way: since in the figures of Prop.~\ref{prop.undefinable}, we have $R_1=R_2$, and we already commented that if $R_1=R_2$, then each of $\boxdot$ and $\boxplus$ becomes a non-contingency operator; moreover, none of the five basic frame properties are definable in a logic with any non-contingency operator as a sole primitive modality~\cite{DBLP:journals/ndjfl/Zolin99,Fanetal:2015}, thus Prop.~\ref{prop.undefinable} obtains. %With a similar argument, we can show that none of the five basic frame properties are definable in $\mathcal{L}(\boxplus)$.
%An alternative method of showing the frame undefinability is sketched as follows: since in the figures of Prop.~\ref{prop.undefinable}, we have $R_1=R_2$, and we already commented that if $R_1=R_2$, then $\boxdot$ becomes a non-contingency operator; moreover, none of the five basic frame properties are definable in a logic with any non-contingency operator as a sole primitive modality, thus Prop.~\ref{prop.undefinable} obtains. With a similar argument, we can show that none of the five basic frame properties are definable in $\mathcal{L}(\boxplus)$.

%\section{The Minimal logic for $\mathcal{L}(\boxplus)$}
\section{Axiomatizations for $\mathcal{L}(\boxplus)$}\label{sec.boxplus-logics}

This section first presents the minimal logic for $\mathcal{L}(\boxplus)$, and shows its soundness and completeness, and then demonstrates that the same logic is also sound and strongly complete with respect to the class of serial bimodal frames.

\subsection{The minimal logic and soundness}

%\subsubsection{Proof system and soundness}

\begin{definition} The minimal logic for $\mathcal{L}(\boxplus)$, denoted ${\bf K^\boxplus}$, consists of the following axioms and inference rules:
\[
\begin{array}{lll}
\text{PC}&&\text{all instances of propositional tautologies}\\
\text{CON}\circledS&&\circledS\phi\land\circledS\psi\to\circledS(\phi\land\psi)\land\circledS(\phi\vee\psi)\\
\text{DIS}\circledS&&\circledS\phi\to\circledS(\phi\vee\psi)\vee\circledS(\phi\land\chi)\\
\text{MP}&&\dfrac{\phi,\phi\to\psi}{\psi}\\
\text{RN}\circledS&&\dfrac{\phi}{\circledS\phi\land\circledS\neg\phi}\\
\text{RE}\circledS&&\dfrac{\phi\lra\psi}{\circledS\phi\lra\circledS\psi}\\
\end{array}
\]
\end{definition}

Notions of deductions and theorems are defined as normal.
%$\dfrac{\psi\to\neg\phi}{\circledS_4(\neg\psi\land\phi)\to\circledS_4\phi}$
%$\dfrac{\psi\to\neg\phi}{\circledS_4\neg\psi\land\circledS_4(\neg\psi\land\phi)\to\circledS_4\phi}$

%It is worth remarking that $\circledS\phi\lra\circledS\neg\phi$, equivalently, $\circledS\phi\to\circledS\neg\phi$ due to the rule $\text{RE}\circledS$, is invalid. For instance, consider a model $\M=\lr{S,R_1,R_2,V}$ in which $S=\{s,t,u\}$, $R_1(s)=\{t\}$ and $R_2(s)=\{u\}$, and $V(p)=\{t\}$. Then it should be easily verified that $s\vDash\circledS p$ but $s\nvDash\circledS\neg p$. This will matter when we would like to see the differences between axiomatizations of $\circledS$-logics and of $\Delta$-logics.

Recall that in the minimal noncontingency logic, the axiom $\Delta\phi\to\Delta(\phi\vee\psi)\vee\Delta(\phi\land\chi)$ (denoted DIS$\Delta$ hereafter) can be replaced with the rule $\dfrac{\phi\to\psi~~~\psi\to\chi}{\Delta\psi\to\Delta\phi\vee\Delta\chi}$~\cite[p.~110]{Humberstone:2002}. This also applies to its $\boxplus$-correspondent; more precisely, the axiom DIS$\circledS$ is replaceable with the rule $\dfrac{\phi\to\psi~~~\psi\to\chi}{\circledS\psi\to\circledS\phi\vee\circledS\chi}$, given the rule RE$\circledS$.

%Note that under the rule RE$\circledS$, the axiom DIS$\circledS$ is replaceable with the rule $\dfrac{\phi\to\psi~~~\psi\to\chi}{\circledS\psi\to\circledS\phi\vee\circledS\chi}$. This is analogous to the case for their $\Delta$-correspondents in the minimal noncontingency logic.

Also, DIS$\Delta$ can be replaced with $\Delta\phi\to\Delta(\phi\vee\psi)\vee\Delta(\neg\phi\vee\chi)$ (called `Kuhn's axiom'), and even with the formula (which is equivalent to Kuhn's axiom) with less district schematic letters $\Delta\phi\to\Delta(\phi\vee\psi)\vee\Delta(\neg\phi\vee\psi)$, see~\cite[pp.~110-111]{Humberstone:2002}. In comparison, the axiom $\text{DIS}\circledS$ cannot be replaced with $\circledS\phi\to\circledS(\phi\vee\psi)\vee\circledS(\neg\phi\vee\chi)$, neither with $\circledS\phi\to\circledS(\phi\vee\psi)\vee\circledS(\neg\phi\vee\psi)$, as illustrated below.
%Recall that in the minimal noncontingency logic, $\Delta\phi\to\Delta(\phi\vee\psi)\vee\Delta(\phi\land\chi)$, the $\Delta$-correspondent of the axiom $\text{DIS}\circledS$ can be replaced with $\Delta\phi\to\Delta(\phi\vee\psi)\vee\Delta(\neg\phi\vee\chi)$ (called `Kuhn's axiom'), and even with the formula (which is equivalent to Kuhn's axiom) with less district schematic letters $\Delta\phi\to\Delta(\phi\vee\psi)\vee\Delta(\neg\phi\vee\psi)$, see~\cite[pp.~110-111]{Humberstone:2002}. In comparison, the axiom $\text{DIS}\circledS$ cannot be replaced with $\circledS\phi\to\circledS(\phi\vee\psi)\vee\circledS(\neg\phi\vee\chi)$, neither with $\circledS\phi\to\circledS(\phi\vee\psi)\vee\circledS(\neg\phi\vee\psi)$, as illustrated below.
\weg{\[
\xymatrix{t_{\overline{p}q}&&s{\ar@(ul,ur)|{1,2}}_{\overline{p}\overline{q}}\ar[ll]|{2}\ar[rr]|{1}&&u_{p\overline{q}}\\}
\]}
\[
\xymatrix{t:{\overline{p}q}&&s{\ar@(ul,ur)|{1}}:{\overline{p}\overline{q}}\ar[ll]|{2}\ar[rr]|{1}&&u:{p\overline{q}}\\}
\]
On one hand, because $R_2(s)\vDash\neg p$, we have $s\vDash\boxplus p$. On the other hand, since $R_1(s)\nvDash p\vee q$ (as $sR_1s$ and $s\nvDash p\vee q$) and $R_2(s)\nvDash\neg (p\vee q)$ (as $sR_2t$ and $t\vDash p\vee q$), it follows that $s\nvDash\boxplus (p\vee q)$; moreover, since $R_1(s)\nvDash \neg p\vee q$ (as $sR_1u$ and $u\vDash p\land\neg q$) and $R_2(s)\nvDash\neg (\neg p\vee q)$ (as $sR_2t$ and $t\vDash \neg p\vee q$), it follows that $s\nvDash\boxplus(\neg p\vee q)$. This indicates that $\circledS p\to\circledS(p\vee q)\vee\circledS(\neg p\vee q)$ is invalid.%\footnote{But note that $\dfrac{\phi\to\psi~~~\psi\to\chi}{\circledS\psi\to\circledS\phi\vee\circledS\chi}$ is valid-preserving.}
%On one hand, because $R_2(s)\vDash\neg p$, we have $s\vDash\boxplus p$. On the other hand, since $R_1(s)\nvDash p\vee q$ (as $sR_1s$ and $s\nvDash p\vee q$) and $R_2(s)\nvDash\neg (p\vee q)$ (as $sR_2t$ and $t\vDash p\vee q$), it follows that $s\nvDash\boxplus (p\vee q)$; moreover, since $R_1(s)\nvDash \neg p\vee q$ (as $sR_1u$ and $u\vDash p\land\neg q$) and $R_2(s)\nvDash\neg (\neg p\vee q)$ (as $sR_2s$ and $s\vDash \neg p\vee q$), it follows that $s\nvDash\boxplus(\neg p\vee q)$.
\begin{proposition}
${\bf K^\boxplus}$ is sound with respect to the class of all bimodal frames.
\end{proposition}

\begin{proof}
We take the validity of $\text{CON}\boxplus$ and $\text{DIS}\boxplus$ as examples. Let $\M=\lr{S,R_1,R_2,V}$ be an arbitrary bimodal model and $s\in S$.

Suppose that $\M,s\vDash\boxplus\phi\land\boxplus\psi$. Then $R_1(s)\vDash\phi$ or $R_2(s)\vDash\neg\phi$, and $R_1(s)\vDash\psi$ or $R_2(s)\vDash\neg\psi$.

If $R_2(s)\vDash\neg\phi$ or $R_2(s)\vDash\neg\psi$, then $R_2(s)\vDash\neg(\phi\land\psi)$; otherwise, that is, if $R_1(s)\vDash\phi$ and $R_1(s)\vDash\psi$, then $R_1(s)\vDash\phi\land\psi$. Thus either $R_1(s)\vDash\phi\land\psi$ or $R_2(s)\vDash\neg(\phi\land\psi)$, and therefore $\M,s\vDash\boxplus(\phi\land\psi)$.

If $R_1(s)\vDash\phi$ or $R_1(s)\vDash\psi$, then $R_1(s)\vDash\phi\vee\psi$; otherwise, that is, if $R_2(s)\vDash\neg\phi$ and $R_2(s)\vDash\neg\psi$, then $R_2(s)\vDash\neg(\phi\vee\psi)$. Thus either $R_1(s)\vDash\phi\vee\psi$ or $R_2(s)\vDash\neg(\phi\vee\psi)$, and therefore $\M,s\vDash\boxplus(\phi\vee\psi)$. Hitherto we have completed the validity of $\text{CON}\boxplus$.

Now suppose that $\M,s\vDash\boxplus\phi$, then $R_1(s)\vDash\phi$ or $R_2(s)\vDash\neg\phi$. If it is the case that $R_1(s)\vDash\phi$, then $R_1(s)\vDash\phi\vee\psi$, which implies that $\M,s\vDash\boxplus(\phi\vee\psi)$; if it is the case that $R_2(s)\vDash\neg\phi$, then $R_2(s)\vDash\neg(\phi\land\chi)$, which entails that $\M,s\vDash\boxplus(\phi\land\chi)$. Therefore, $\M,s\vDash\boxplus(\phi\vee\psi)\vee\boxplus(\phi\land\chi)$. Hitherto we have completed the validity of $\text{DIS}\boxplus$.
\end{proof}

%\subsubsection{Completeness}
\subsection{Completeness}

This part deals with the completeness of ${\bf K^\boxplus}$. We adopt the standard canonical model construction. However, a tricky thing is how to define two suitable canonical relations to handle the operator $\circledS$.

\begin{definition}\label{def.cm-boxplus} The {\em canonical model} for ${\bf K^\boxplus}$ is a tuple $\M^c=\lr{S^c,R^c_1,R^c_2,V^c}$, where
\begin{itemize}
\item $S^c=\{s\mid s\text{ is a maximal }{\bf K^\boxplus}\text{-consistent set}\}$;
\item $sR^c_1t$ iff $\lambda_1(s)\subseteq t$, where $\lambda_1(s)=\{\phi\mid \circledS(\phi\vee\psi)\in s\text{ for all }\psi\}$;
\item $sR^c_2u$ iff $\lambda_2(s)\subseteq u$, where $\lambda_2(s)=\{\phi\mid \circledS(\neg\phi\land \chi)\in s\text{ for all }\chi\}$;
%\item $sR^c_1t$ iff for all $\phi$, if $\circledS(\phi\vee\psi)\in s$ for all $\psi$, then $\phi\in t$.
%\item $sR^c_2t$ iff for all $\phi$, if $\circledS(\neg\phi\land \chi)\in s$ for all $\chi$, then $\phi\in t$.%\[\begin{array}{lll}
%sR^c_1t& \text{iff}& \text{for all }\phi, \text{ if }\circledS(\phi\vee\psi)\in s \text{ for all }\psi, \text{ then }\phi\in t.\\
%sR^c_2t& \text{iff}& \text{for all }\phi, \text{ if }\circledS(\neg\phi\land\chi)\in s\text{ for all }\chi,\text{ then }\phi\in t.\\
%\end{array}
%\]
\item $V^c(p)=\{s\in S^c\mid p\in s\}$.
\end{itemize}
\end{definition}

As mentioned, the semantics of $\Delta$ is a special case of the semantics of $\boxplus$ when $R_1=R_2$. %The scrupulous reader may wonder if
In that case, we should have $R^c_1=R^c_2$. Indeed this is true, since in that case, $\circledS\phi\lra\circledS\neg\phi$ is valid, and then the definition of $R_2^c$ is equivalent to that ``$\text{for all }\phi, \text{ if }\circledS(\phi\vee\psi)\in s\text{ for all }\psi,\text{ then }\phi\in t$'', that is, the definition of $R^c_1$. And in this way, we obtain the canonical relation defined in~\cite{DBLP:journals/ndjfl/Kuhn95} as a special case.
%Note that if $\Box_2=\Box_1$, then we have $R_1=R_2$, and thus we should have $R^c_1=R^c_2$. Indeed this is the case, since in that case, $\circledS\phi\lra\circledS\neg\phi$ is valid, and then the definition of $R_2^c$ is equivalent to that ``$\text{for all }\phi, \text{ if }\circledS(\phi\vee\psi)\in s\text{ for all }\psi,\text{ then }\phi\in t$.'', that is, to the definition of $R^c_1$. And therefore, we obtain the canonical relation defined in~\cite{DBLP:journals/ndjfl/Kuhn95}.

%Let $\lambda_1(s)=\{\phi\mid \circledS(\phi\vee\psi)\in s\text{ for all }\psi\}$ and $\lambda_2(s)=\{\phi\mid \circledS(\neg\phi\land \chi)\in s\text{ for all }\chi\}$.

Let us look at the properties of the two functions $\lambda_1$ and $\lambda_2$.
\begin{proposition}\label{prop.import} Let $s\in S^c$. Then
\begin{enumerate}
\item[(a)] $\lambda_1(s)\cap\lambda_2(s)$ is nonempty. Consequently, $\lambda_1(s)$ and $\lambda_2(s)$ are both nonempty.
\item[(b)] $\lambda_1(s)$ and $\lambda_2(s)$ are both closed under conjunction. That is, if $\phi_1,\phi_2\in\lambda_1(s)$, then $\phi_1\land\phi_2\in\lambda_1(s)$, and similarly for $\lambda_2(s)$. Consequently, $\lambda_1(s)$ and $\lambda_2(s)$ are both closed under finite conjunctions.
\item[(c)] If $\phi\in\lambda_1(s)$ and $\vdash\phi\to\delta$, then $\delta\in\lambda_1(s)$, and similarly for $\lambda_2(s)$.
\item[(d)] $\circledS\phi\in s$ iff either $\phi\in \lambda_1(s)$ or $\neg\phi\in\lambda_2(s)$.
%\item[(d)] If $\circledS\phi\in s$, then either $\phi\in \lambda_1(s)$ or $\neg\phi\in\lambda_2(s)$.
\end{enumerate}
\end{proposition}

\begin{proof}\
\begin{enumerate}
\item[(a)] Since $\vdash \top$, then applying the rule $\text{RN}\circledS$, we have $\vdash \circledS\top\land \circledS\neg\top$. By $\text{RE}\boxplus$, $\circledS(\top\vee\psi)\in s$ for all $\psi$ and $\circledS(\neg\top\land\chi)\in s$ for all $\chi$. Therefore, $\top\in \lambda_1(s)\cap\lambda_2(s)$.
\item[(b)] Suppose that $\phi_1,\phi_2\in\lambda_1(s)$, then $\circledS(\phi_1\vee\psi)\in s$ and $\circledS(\phi_2\vee\psi)\in s$ for all $\psi$. Then $\circledS(\phi_1\vee\psi)\land \circledS(\phi_2\vee\psi)\in s$. Using the axiom $\text{CON}\circledS$, we obtain $\circledS((\phi_1\vee\psi)\land(\phi_2\vee\psi))\in s$. Then applying the rule $\text{RE}\circledS$, we infer that $\circledS((\phi_1\land\phi_2)\vee\psi)\in s$. Since $\psi$ is arbitrary, we now conclude that $\phi_1\land\phi_2\in\lambda_1(s)$.

    Assume that $\phi_1,\phi_2\in\lambda_2(s)$, then $\circledS(\neg\phi_1\land\chi)\in s$ and $\circledS(\neg\phi_2\land\chi)\in s$ for all $\chi$. Then $\circledS(\neg\phi_1\land\chi)\land \circledS(\neg\phi_2\land\chi)\in s$. Using the axiom $\text{CON}\circledS$, we infer $\circledS((\neg\phi_1\land\chi)\vee (\neg\phi_2\land\chi))\in s$. Now applying the rule $\text{RE}\circledS$, we obtain $\circledS((\neg(\phi_1\land\phi_2)\land\chi))\in s$. Since $\chi$ is arbitrary, we now conclude that $\phi_1\land\phi_2\in \lambda_2(s)$.
\item[(c)] Suppose $\phi\in\lambda_1(s)$ and $\vdash \phi\to\delta$, to show $\delta\in\lambda_1(s)$. By supposition, it follows that $\vdash\phi\vee\delta\lra\delta$ and $\circledS(\phi\vee\psi)\in s$ for all $\psi$. Then $\vdash\phi\vee(\delta\vee\psi)\lra\delta\vee\psi$. Applying the rule $\text{RE}\circledS$, we derive $\vdash\circledS(\phi\vee(\delta\vee\psi))\lra\circledS(\delta\vee\psi)$. Since $\circledS(\phi\vee(\delta\vee\psi))\in s$, we derive that $\circledS(\delta\vee\psi)\in s$. Since $\psi$ is arbitrary, $\delta\in \lambda_1(s)$.

    Now assume that $\phi\in\lambda_2(s)$ and $\vdash \phi\to\delta$, to show $\delta\in\lambda_2(s)$. Since $\phi\in\lambda_2(s)$, it follows that $\circledS(\neg\phi\land\chi)\in s$ for all $\chi$. Since $\vdash \phi\to\delta$, it follows that $\vdash\neg\phi\land\neg\delta\lra\neg\delta$, and thus $\vdash\neg\phi\land(\neg\delta\land\chi)\lra\neg\delta\land\chi$. Applying the rule $\text{RE}\circledS$, we obtain $\vdash\circledS(\neg\phi\land(\neg\delta\land\chi))\lra\circledS(\neg\delta\land\chi)$. Since $\circledS(\neg\phi\land(\neg\delta\land\chi))\in s$, we get $\circledS(\neg\delta\land\chi)\in s$. Since $\chi$ is arbitrary, $\delta\in \lambda_2(s)$.
\item[(d)] Suppose by contraposition that $\phi\notin \lambda_1(s)$ and $\neg\phi\notin\lambda_2(s)$. Then $\circledS(\phi\vee\psi)\notin s$ for some $\psi$, and $\circledS(\neg\neg\phi\land\chi)\notin s$ for some $\chi$, namely $\circledS(\phi\land\chi)\notin s$. Using the axiom $\text{DIS}\circledS$, we obtain immediately $\circledS\phi\notin s$.

    Conversely, assume that either $\phi\in\lambda_1(s)$ or $\neg\phi\in\lambda_2(s)$. Then {\em either} $\boxplus(\phi\vee\psi)\in s$ for all $\psi$ {\em or} $\boxplus(\phi\land\chi)\in s$ for all $\chi$. Then either case implies that $\boxplus\phi\in s$: in the first case, letting $\psi=\bot$, by RE$\boxplus$ we obtain $\boxplus\phi\in s$; in the second case, let $\chi=\top$, by RE$\boxplus$ again, we infer that $\boxplus\phi\in s$. Therefore, $\boxplus\phi\in s$.
\end{enumerate}
\end{proof}

With the above results in preparation, we can obtain the following truth lemma.
\begin{lemma}
For all $s\in S^c$, for all $\phi\in \mathcal{L}(\boxplus)$, we have
$$\phi\in s\iff\M^c,s\vDash\phi.$$
\end{lemma}

\begin{proof}
By induction on $\phi$. The only nontrivial case is $\boxplus\phi$.

Assume for reductio that $\circledS\phi\in s$ but $\M^c,s\nvDash\circledS\phi$. By induction hypothesis, there is a $t$ such that $sR_1^ct$ and $\phi\notin t$, and there is a $u$ such that $sR_2^cu$ and $\neg\phi\notin u$. Then by definitions of $R^c_1$ and $R^c_2$, we can obtain that $\phi\notin \lambda_1(s)$ and $\neg\phi\notin \lambda_2(s)$. This contradicts the supposition that $\circledS\phi\in s$ and Prop.~\ref{prop.import}(d).
%Assume for reductio that $\circledS\phi\in s$ but $\M^c,s\nvDash\circledS\phi$. By induction hypothesis, this means that there is a $t$ such that $sR_1^ct$ and $\phi\notin t$, and there is a $u$ such that $sR_2^cu$ and $\phi\in u$. Then by the definition of $R_1^c$, we infer that $\circledS(\phi\vee\psi)\notin s$ for some $\psi$; by the definition of $R_2^c$, we derive that $\circledS(\phi\land\chi)\notin s$ for some $\chi$. Thus by axiom ..., $\circledS\phi\notin s$, contrary to the assumption.

Conversely, suppose $\circledS\phi\notin s$, we need to find two states $t$ and $u$ in $S^c$ such that $sR^c_1t$ and $\phi\notin t$, and $sR^c_2u$ and $\phi\in u$. For this, we first show that
\begin{enumerate}
\item[(1)] $\lambda_1(s)\cup\{\neg\phi\}$ is consistent, and
\item[(2)] $\lambda_2(s)\cup\{\phi\}$ is consistent.
%\item[(1)] $\{\chi\mid \circledS(\chi\vee\delta)\in s\text{ for all }\delta\}\cup\{\neg\phi\}$ is consistent, and
%\item[(2)] $\{\psi\mid \circledS(\neg\psi\land\chi)\in s\text{ for all }\chi\}\cup\{\phi\}$ is consistent.
\end{enumerate}

If (1) does not hold, then there exist $\chi_1,\cdots,\chi_n\in\lambda_1(s)$\footnote{Prop.~\ref{prop.import}(a) provides the nonempty of $\lambda_1(s)$.} such that $\vdash\chi_1\land\cdots\land\chi_n\to\phi$. Since $\chi_1,\cdots,\chi_n\in\lambda_1(s)$, from Prop.~\ref{prop.import}(b) it follows that $\chi_1\land\cdots\land\chi_n\in\lambda_1(s)$. Then due to Prop.~\ref{prop.import}(c), we have $\phi\in\lambda_1(s)$, by Prop.~\ref{prop.import}(d) we conclude that $\circledS\phi\in s$, contrary to the supposition.
%If (1) does not hold, then there exist $\chi_1,\cdots,\chi_n\in\lambda_1(s)$\footnote{Prop.~\ref{prop.import}(a) provides the nonempty of $\lambda_1(s)$.} such that $\circledS(\chi_i\vee\delta)\in s$ for all $\delta$ and $i\in[1,n]$ and $\vdash\chi_1\land\cdots\land\chi_n\to\phi$. Thus $\circledS(\chi_i\vee\phi)\in s$. Using Prop.~\ref{prop.admissible-1}, we obtain $\vdash\circledS(\chi_1\vee\phi)\land\cdots\land\circledS(\chi_n\vee\phi)\to\circledS\phi$. Therefore, $\circledS\phi\in s$, contrary to the supposition.

If (2) does not hold, then there are $\psi_1,\cdots,\psi_m\in\lambda_2(s)$\footnote{Again, Prop.~\ref{prop.import}(a) provides the nonempty of $\lambda_2(s)$.} such that $\vdash\psi_1\land\cdots\land\psi_m\to\neg\phi$. Since $\psi_1,\cdots,\psi_m\in\lambda_2(s)$, it follows that $\psi_1\land\cdots\land\psi_m\in\lambda_2(s)$ from Prop.~\ref{prop.import}(b). Then thanks to Prop.~\ref{prop.import}(c), we infer that $\neg\phi\in\lambda_2(s)$, by Prop.~\ref{prop.import}(d) again, we derive that $\circledS\phi\in s$, which contradicts the supposition again.

Then by Lindenbaum's Lemma, we are done.
\end{proof}

Now it is a standard exercise to show that ${\bf K^\boxplus}$ is the minimal logic of $\mathcal{L}(\boxplus)$.
\begin{theorem}
${\bf K^\boxplus}$ is sound and strongly complete with respect to the class of all bimodal frames.
\end{theorem}

\subsection{The serial logic}

In this section, we show that ${\bf K^\boxplus}$ is also the serial logic of $\mathcal{L}(\boxplus)$, that is to say, ${\bf K^\boxplus}$ is sound and strongly complete with respect to the class of serial bimodal frames. For this, if $R^c_1$ and $R^c_2$ in Def.~\ref{def.cm-boxplus} are serial, then we are done. We first have the following key observation.

\begin{proposition}\label{prop.serial-equiv}
Define $\M^c$ as in Def.~\ref{def.cm-boxplus} and $s\in S^c$. Then the following conditions are equivalent:
\begin{enumerate}
\item\label{prop.serial-1} $R^c_1(s)\neq \emptyset$. %and $R^c_2$ are both serial
\item\label{prop.serial-2} $\bot\notin \lambda_1(s)$.
\item\label{prop.serial-3} $\boxplus\psi\notin s$ for some $\psi$.
\item\label{prop.serial-4} $\bot\notin \lambda_2(s)$.
\item\label{prop.serial-5} $R^c_2(s)\neq \emptyset$.
\end{enumerate}
\end{proposition}

\begin{proof}
We show $\ref{prop.serial-1}\iff \ref{prop.serial-2}$, $\ref{prop.serial-2}\iff\ref{prop.serial-3}$, $\ref{prop.serial-3}\iff \ref{prop.serial-4}$, and $\ref{prop.serial-4}\iff\ref{prop.serial-5}$.
\end{proof}

\begin{proof}
$\ref{prop.serial-1}\iff \ref{prop.serial-2}$: suppose towards contradiction that $R^c_1(s)\neq \emptyset$ but $\bot\in \lambda_1(s)$. Then $sR^c_1t$ for some $t\in S^c$, that is, $\lambda_1(s)\subseteq t$, and therefore $\bot\in t$: a contradiction. Conversely, assume that $\bot\notin \lambda_1(s)$, we need to show that $s$ has a $R^c_1$-successor. It suffices to show that $\lambda_1(s)$ is consistent. If not, there exists $\phi_1,\cdots,\phi_n\in\lambda_1(s)$ such that $\vdash\phi_1\land\cdots\land\phi_n\to\bot$. Using items (b) and (c) of Prop.~\ref{prop.import}, we can derive that $\bot\in\lambda_1(s)$, which is contrary to the assumption.

$\ref{prop.serial-2}\iff\ref{prop.serial-3}$: Suppose by contraposition that $\boxplus\psi\in s$ for all $\psi$. Since $\vdash\psi\lra\bot\vee\psi$, by RE$\boxplus$, it follows that $\vdash\boxplus\psi\lra\boxplus(\bot\vee\psi)$, and then $\boxplus(\bot\vee\psi)\in s$, and therefore $\bot\in \lambda_1(s)$. Conversely, assume that $\bot\in \lambda_1(s)$, then $\boxplus(\bot\vee\psi)\in s$ for all $\psi$, and thus $\boxplus\psi\in s$ for all $\psi$.

$\ref{prop.serial-3}\iff \ref{prop.serial-4}$: similar to the proof of $\ref{prop.serial-2}\iff\ref{prop.serial-3}$.

$\ref{prop.serial-4}\iff\ref{prop.serial-5}$: similar to the proof of $\ref{prop.serial-1}\iff \ref{prop.serial-2}$.
\end{proof}

\begin{corollary}
Define $\M^c$ as in Def.~\ref{def.cm-boxplus}. Then the following conditions are equivalent:
\begin{enumerate}
\item $R^c_1$ is serial. %and $R^c_2$ are both serial
\item $\bot\notin \lambda_1(s)$ for any $s\in S^c$.
\item $\boxplus\psi\notin s$ for any $s\in S^c$ and for some $\psi$.
\item $\bot\notin \lambda_2(s)$ for any $s\in S^c$.
\item $R^c_2$ is serial.
\end{enumerate}
\end{corollary}

As we cannot exclude the possibility that $\boxplus\psi\in s$ for some $s\in S$ and for all $\psi$, by the above result, we cannot provide that $R^c_1$ and $R^c_2$ are serial. We call such states $s$ `endpoints'. By Prop.~\ref{prop.serial-equiv}, $s$ has neither $R^c_1$-successors nor $R^c_2$-successors.

We handle these endpoints by using a similar strategy of `reflexivizing the arrows in the canonical model' used for showing the completeness of serial contingency logic in~\cite{Humberstone95,Fanetal:2015}. In detail, define $\M^{\bf D}=\lr{S^c,R^{\bf D}_1,R^{\bf D}_2,V^c}$ as $\M^c$ in Def.~\ref{def.cm-boxplus}, except that $R^{\bf D}_i=R^c_i\cup\{(s,s)\mid s\text{ is an endpoint}\}$. It should be obvious that $\M^{\bf D}$ is serial. Moreover, the truth values of $\mathcal{L}(\boxplus)$-formulas are invariant under the model transformation: for all $s\in S^c$, by Prop.~\ref{prop.serial-equiv}, $R^c_1(s)\neq \emptyset$ iff $R^c_2(s)\neq \emptyset$. If $s$ is an endpoint, then as $R^c_1(s)=R^c_2(s)=\emptyset$, it holds vacuously that $\M^c,s\vDash\boxplus\phi$; since $s\vDash\phi$ or $s\vDash\neg\phi$ and $R_1^{\bf D}(s)=R_2^{\bf D}(s)=\{s\}$, we have also that $\M^{\bf D},s\vDash\boxplus\phi$. If $s$ has both $R^c_1$- and $R^c_2$-successors, then it is clear that $\M^c,s\vDash\boxplus\phi$ iff $\M^{\bf D},s\vDash\boxplus\phi$, as desired. Consequently,

%Note that the strategy `reflexivizing the endpoints' in~\cite{Humberstone95,Fanetal:2015} does not work here. Consider the canonical model $\M^c=\lr{S^c,R^c_1,R^c_2,V^c}$ with an endpoint $s$.

%One may easily verify that $\M^c,s\vDash\circledS p$. If we adopt the strategy `reflexivizing the $R_1$-endpoints' (in the current context, by adding the $R_1$-arrow from $s$ to itself) and denote the obtained model as $\M'$, then one may show that $\M',s\nvDash\circledS p$.
%Consider a (non-$R_1$-serial) model $\M=\lr{S,R_1,R_2,V}$ where $S=\{s,t\}$, $R_1(s)=\emptyset$, $R_2(s)=\{t\}$ and $R_i(t)=\{t\}$ for $i=1,2$, and $V(p)=\{t\}$. One may easily verify that $\M,s\vDash\circledS p$. If we adopt the strategy `reflexivizing the $R_1$-endpoints' (in the current context, by adding the $R_1$-arrow from $s$ to itself) and denote the obtained model as $\M'$, then one may show that $\M',s\nvDash\circledS p$.

\begin{theorem}\label{thm.boxplus-serial}
${\bf K^\boxplus}$ is sound and strongly complete with respect to the class of serial bimodal frames.
\end{theorem}

\section{Axiomatizations for $\mathcal{L}(\boxdot)$}\label{sec.boxdot-logics}

This section first provides the minimal logic for $\mathcal{L}(\boxdot)$ and shows its soundness and completeness, then explores its extensions over special frames.

\subsection{Minimal logic}

\begin{definition}
The minimal logic of $\mathcal{L}(\boxdot)$, denoted ${\bf K^\boxdot}$, consists of the following axioms and inference rules:
\[
\begin{array}{ll}
\text{PC}&\text{All instances of propositional tautologies}\\
\boxdot\text{T}&\boxdot\top\\
\boxdot\text{EQU}&\boxdot\phi\lra\boxdot\neg\phi\\
%\boxdot\text{CON}&\boxdot\phi\land\boxdot\psi\to\boxdot(\phi\land\psi)\land\boxdot(\phi\vee\psi)\\
\boxdot\text{CON}&\boxdot\phi\land\boxdot\psi\to\boxdot(\phi\land\psi)\\
%\boxdot\phi\to\boxdot(\phi\vee\psi)\vee\boxdot(\phi\land\chi)$
\boxdot\text{DIS}&\boxdot\phi\to\boxdot(\phi\vee\psi)\vee\boxdot(\neg\phi\vee\chi)\\
\text{MP}&\dfrac{\phi~~~\phi\to\psi}{\psi}\\
%$\dfrac{\phi}{\boxdot\phi}$
\text{RE}\boxdot&\dfrac{\phi\lra\psi}{\boxdot\phi\lra\boxdot\psi}\\
\end{array}
\]
\end{definition}

\weg{\begin{definition}
Define a translation $t$ from $\mathcal{L}(\boxdot)$ to $\mathcal{L}(\Delta)$ as follows:
\[
\begin{array}{lll}
t(p)&=&p\\
t(\neg\phi)&=&\neg t(\phi)\\
t(\phi\land\psi)&=&t(\phi)\land t(\psi)\\
t(\boxdot\phi)&=&\Delta t(\phi)\\
\end{array}
\]
\end{definition}

\begin{proposition}
For all $\phi\in\mathcal{L}(\boxdot)$, then $\vdash_{{\bf K^\boxdot}}\phi~\iff~\vdash_{{\bf K^\Delta}} t(\phi).$
\end{proposition}

Define a $\Delta$-model $\M=\lr{S,R,V}$, we can construct a $\boxdot$-model $\M=\lr{S^\star,R^\star_1,R^\star_2,V^\star}$, where
\begin{itemize}
\item $S^\star=S$
\item $R^\star_1=R^\star_2=R$
\item $V^\star=V$.
\end{itemize}

\begin{proposition}
Let $\M=\lr{S,R,V}$ be a $\Delta$-model. For every $\phi\in \mathcal{L}(\boxdot)$ and for every $s\in S$, we have that $\M,s\vDash t(\phi)$ iff $\M^\star,s\vDash\phi$.
\end{proposition}

\begin{proof}
By induction on $\phi$. The nontrivial case is $\boxdot\phi$, that is to show, $\M,s\vDash \Delta t(\phi)$ iff $\M^\star,s\vDash\boxdot\phi$.

If $\M,s\nvDash\Delta t(\phi)$, then there exist $t,u\in S$ such that $sRt$ and $sRu$ and $\M,t\vDash t(\phi)$ and $\M,u\nvDash t(\phi)$. According to the construction of $\M^\star$ and the induction hypothesis, this means that $sR^\star_1t$ and $sR^\star_2u$ and $\M^\star,t\vDash \phi$ and $\M^\star,u\nvDash\phi$. Therefore, $\M^\star,s\nvDash\boxdot\phi$. The converse is similar.
\end{proof}

\begin{theorem}[Completeness of ${\bf K^\boxdot}$]
${\bf K^\boxdot}$ is (weakly) complete
\end{theorem}}

The proposition below will be used in Prop.~\ref{prop.boxdot-property}.

\begin{proposition}\label{prop.admissible-provable}
The rule $\dfrac{\phi\to\psi}{\boxdot\phi\land\boxdot(\psi\to\phi)\to \boxdot\psi}$, denoted $\text{wM}\boxdot$, is derivable in ${\bf K^\boxdot}$.
\end{proposition}

\begin{proof}
Suppose that $\vdash\phi\to\psi$, then $\vdash\phi\lra(\phi\land\psi)$. By $\text{RE}\boxdot$, we have $\vdash\boxdot\phi\lra\boxdot(\phi\land\psi)$. By axiom $\boxdot\text{CON}$, $\vdash\boxdot(\phi\to\neg\psi)\land\boxdot(\neg\phi\to\neg\psi)\to \boxdot((\phi\to\neg\psi)\land(\neg\phi\to\neg\psi))$. Since $\vdash(\phi\to\neg\psi)\land(\neg\phi\to\neg\psi)\lra \neg\psi$, by $\text{RE}\boxdot$ it follows that $\vdash\boxdot((\phi\to\neg\psi)\land(\neg\phi\to\neg\psi))\lra \boxdot\neg\psi$. Using PC, $\boxdot\text{EQU}$ and $\text{RE}\boxdot$, we obtain $\vdash\boxdot(\phi\land\psi)\lra \boxdot(\phi\to\neg\psi)$ and $\vdash\boxdot(\psi\to\phi)\lra \boxdot(\neg\phi\to\neg\psi)$ and $\vdash\boxdot\neg\psi\lra\boxdot\psi$, and therefore $\vdash\boxdot\phi\land\boxdot(\psi\to\phi)\to \boxdot\psi$.
\end{proof}

\begin{proposition}
${\bf K^\boxdot}$ is sound with respect to the class of all bimodal frames.
\end{proposition}

\begin{proof}
We only show the validity of axioms $\boxdot\text{CON}$ and $\boxdot\text{DIS}$. Let $\M=\lr{S,R_1,R_2,V}$ be an arbitrary bimodal model and $s\in S$.

For the validity of $\boxdot\text{CON}$, suppose that $\M,s\vDash\boxdot\phi\land\boxdot\psi$, then for all $t,u$ such that $sR_1t$ and $sR_2u$, we have that ($t\vDash\phi$ iff $u\vDash\phi$), and also that ($t\vDash\psi$ iff $u\vDash\psi$), thus $t\vDash\phi\land\psi$ iff ($t\vDash\phi$ and $t\vDash\psi$) iff ($u\vDash\phi$ and $u\vDash\psi$) iff $u\vDash\phi\land\psi$, and thus $s\vDash\boxdot(\phi\land\psi)$.

For the validity of $\boxdot\text{DIS}$, suppose that $\M,s\vDash\boxdot\phi$, then for all $t,u$ such that $sR_1t$ and $sR_2u$, we have that ($t\vDash\phi$ iff $u\vDash\phi$). If $\phi$ is true at both $t$ and $u$, then so is $\phi\vee\psi$; if $\phi$ is false at both $t$ and $u$, then $\neg\phi$ is true at both points, and so is $\neg\phi\vee\chi$. Therefore, $\M,s\vDash\boxdot(\phi\vee\psi)\vee\boxdot(\neg\phi\vee\chi)$, as desired.
\end{proof}

\medskip

In the remainder of this subsection, we show the strong completeness of ${\bf K^\boxdot}$. The following canonical model is inspired by that of the minimal noncontingency logic in~\cite{Fanetal:2015} and the similarity between $\boxdot$-axioms and $\Delta$-axioms.
\begin{definition}\label{def.cm-boxdot} A tuple $\M^c=\{S^c,R^c_1,R^c_2,V^c\}$ is the {\em canonical model} of ${\bf K^{\boxdot}}$, if
\begin{itemize}
\item $S^c=\{s\mid s\text{ is a maximal }{\bf K^{\boxdot}}\text{-consistent set}\}$,
\item For $i\in\{1,2\}$, $sR^c_it$ iff there exists $\chi$ such that
\begin{enumerate}
\item $\neg\boxdot\chi\in s$ and
\item for all $\phi$, if $\boxdot\phi\land\boxdot(\chi\to\phi)\in s$, then $\phi\in t$.
\end{enumerate}
\item $V^c(p)=\{s\in S^c\mid p\in s\}$.
\end{itemize}
\end{definition}

Note that $R^c_1=R^c_2$. This fact will make our proofs much more convenient.

\weg{Given that $\neg\boxdot\chi\in s$ for some $\chi$, we define $\lambda(s)=\{\phi\mid \boxdot\phi\land\boxdot(\chi\to\phi)\}$; otherwise, $\lambda(s)$ is undefined.

NOT WORK!\begin{definition}
Let $s$ be a maximal consistent set. If $\boxdot\alpha\in s$ for all $\alpha$, then define $\lambda(s)=\emptyset$. Otherwise, we enumerate all the elements of $\{\alpha\mid \neg\boxdot\alpha\in s\}$ and denote the first element as $\eta$ and define $\lambda(s)=\{\phi\mid \boxdot\phi\land\boxdot(\eta\to\phi)\in s\}$.
\end{definition}

\begin{definition} A triple $\M^c=\{S^c,R^c_1,R^c_2,V^c\}$ is the canonical model, if
\begin{itemize}
\item $S^c=\{s\mid s\text{ is a maximal consistent set}\}$,
\item $sR^c_it$ iff $\lambda(s)\subseteq t$, where $i=1,2$,
\item $V^c(p)=\{s\in S^c\mid p\in s\}$.
\end{itemize}
\end{definition}}

\weg{\begin{definition} A triple $\M^c=\{S^c,R^c_1,R^c_2,V^c\}$ is the canonical model, if
\begin{itemize}
\item $S^c=\{s\mid s\text{ is a maximal consistent set}\}$,
\item $sR^c_it$ iff $\lambda_\chi(s)\subseteq t$ for some $\neg\boxdot\chi\in s$, where $i=1,2$,
\item $V^c(p)=\{s\in S^c\mid p\in s\}$.
\end{itemize}
\end{definition}}

\begin{proposition}\label{prop.boxdot-property}
Let $s\in S^c$, $\boxdot\phi\notin s$ and $\Gamma(s)=\{\psi\mid \boxdot\psi\land\boxdot(\phi\to\psi)\in s\}$. Then
\begin{enumerate}
\item\label{boxdot1} $\Gamma(s)$ is nonempty.
\item\label{boxdot2} If $\psi,\chi\in\Gamma(s)$, then $\psi\land\chi\in\Gamma(s)$.
%\item If $\psi\in\Gamma(s)$ and $\vdash\psi\to\phi$, then $\boxdot\phi\in s$.
\item\label{boxdot3} If $\psi\in\Gamma(s)$, then $\nvdash\psi\to\phi$.
\item\label{boxdot4} $\Gamma(s)\cup\{\phi\}$ and $\Gamma(s)\cup\{\neg\phi\}$ are both consistent.
\end{enumerate}
\end{proposition}

\begin{proof} Suppose that the preconditions hold. Then $\boxdot\neg\phi\notin s$.
\begin{enumerate}
\item Straightforward because $\vdash\boxdot\top$.
\item Assume that $\psi,\chi\in\Gamma(s)$, then $\boxdot\psi\land\boxdot(\phi\to\psi)\in s$ and $\boxdot\chi\land\boxdot(\phi\to\chi)\in s$. By axiom $\boxdot\text{CON}$, it follows that $\boxdot(\psi\land\chi)\land\boxdot(\phi\to\psi\land\chi)\in s$, and therefore $\psi\land\chi\in\Gamma(s)$.
\item Assume for reductio that $\psi\in\Gamma(s)$ and $\vdash\psi\to\phi$. Then $\boxdot\psi\land\boxdot(\phi\to\psi)\in s$ and $\vdash\boxdot\psi\land\boxdot(\phi\to\psi)\to\boxdot\phi$ (by the rule $\text{wM}\boxdot$ in Prop.~\ref{prop.admissible-provable}), and therefore $\boxdot\phi\in s$, which contradicts the supposition that $\boxdot\phi\notin s$.
\item Assume that $\Gamma(s)\cup\{\phi\}$ is inconsistent, then there exists $\psi_1,\cdots,\psi_m\in\Gamma(s)$ (\ref{boxdot1} provides the nonempty of $\Gamma(s)$) such that $\vdash\psi_1\land\cdots\land\psi_m\to\neg\phi$. By application of~\ref{boxdot2} for $m-1$ times, we can obtain that $\psi_1\land\cdots\land\psi_m\in\Gamma(s)$, which contradicts~\ref{boxdot3}. Thus $\Gamma(s)\cup\{\phi\}$ is consistent. Similarly, we can conclude that $\Gamma(s)\cup\{\neg\phi\}$ is consistent.
\end{enumerate}
\end{proof}

\begin{lemma}[Truth Lemma for ${\bf K^{\boxdot}}$]\label{lem.truth-lemma-k}
For all $s\in S^c$, for all $\phi\in\mathcal{L}(\boxdot)$, we have
$$\M^c,s\vDash\phi\text{ iff }\phi\in s.$$
\end{lemma}

\begin{proof}
By induction on $\phi\in\mathcal{L}(\boxdot)$. The nontrivial case is $\boxdot\phi$.

Suppose that $\boxdot\phi\in s$ (thus $\boxdot\neg\phi\in s$), to show that $\M^c,s\vDash\boxdot\phi$. If not, by induction hypothesis, there exist $t,u\in S^c$ such that $sR^c_1t$ and $sR^c_2u$ and it is {\em not} the case that ($\phi\in t$ iff $\phi\in u$). W.l.o.g. we may assume\footnote{This is because $R^c_1=R^c_2$.} that $\phi\in t$ but $\phi\notin u$. From $sR^c_1t$, it follows that there exists $\chi$ such that $\neg\boxdot\chi\in s$ and (1) for all $\phi$, if $\boxdot\phi\land\boxdot(\chi\to\phi)\in s$, then $\phi\in t$. Since $\neg\phi\notin t$ and $\boxdot\neg\phi\in s$, by (1) we have $\boxdot(\chi\to\neg\phi)\notin s$, namely $\boxdot(\neg\phi\vee\neg\chi)\notin s$.  Similarly, from $sR^c_2u$ and $\phi\notin u$, we can show that for some $\psi$, $\boxdot(\psi\to\phi)\notin s$, that is, $\boxdot(\phi\vee\neg\psi)\notin s$. Now by axiom $\boxdot\text{DIS}$, we obtain that $\boxdot\phi\notin s$, which is contrary to the supposition.

Conversely, assume that $\boxdot\phi\notin s$, we need to find two states $t,u\in S^c$ such that $sR^c_1t$ and $sR^c_2u$ and it is {\em not} the case that ($\phi\in t$ iff $\phi\in u$). Define $\Gamma(s)$ as in Prop.~\ref{prop.boxdot-property}. By Prop.~\ref{prop.boxdot-property}.\ref{boxdot4}, $\Gamma(s)\cup\{\phi\}$ and $\Gamma(s)\cup\{\neg\phi\}$ are both consistent. Then by Lindenbaum's Lemma, there are two states $t,u\in S^c$ such that $sR^c_1t$ and $sR^c_2u$ such that $\phi\in t$ and $\phi\notin u$, and thus it is {\em not} the case that ($\phi\in t$ iff $\phi\in u$), as desired.
\end{proof}

The strong completeness is now a standard exercise.
\begin{theorem}\label{thm.k-completeness}
${\bf K^{\boxdot}}$ is sound and strongly complete with respect to the class of all bimodal frames.
\end{theorem}

\subsection{Extensions}\label{sec.extensions}

In this section, we study the axiomatizations of $\mathcal{L}(\boxdot)$ over special frames. The following table lists extra axioms and proof systems, and the frame properties that the corresponding systems characterize.
\[
\begin{array}{|llll|}
\hline
\text{Notation}&\text{Axioms}&\text{Systems}&\text{Properties}\\
\hline\hline
&&{\bf K^\boxdot}&\text{seriality}\\
\hline
\boxdot\text{T}&\phi\to[\boxdot\phi\to(\boxdot(\phi\to\psi)\to\boxdot\psi)]&{\bf T^\boxdot}={\bf K^\boxdot}+\boxdot\text{T}&\text{reflexivity}\\
\boxdot\text{B}&\phi\to\boxdot((\boxdot\phi\land\boxdot(\phi\to\psi)\land\neg\boxdot\psi)\to\chi)&{\bf B^\boxdot}={\bf K^\boxdot}+\boxdot\text{B}&\text{symmetry}\\
\boxdot4&\boxdot\phi\to\boxdot(\boxdot\phi\vee\psi)&{\bf K4^\boxdot}={\bf K^\boxdot}+\boxdot4&qt\& pt\\
\boxdot5&\neg\boxdot\phi\to\boxdot(\neg\boxdot\phi\vee\psi)&{\bf K5^\boxdot}={\bf K^\boxdot}+\boxdot5&qe\& pe\\
\hline
\end{array}
\]
%Where $qt$ abbreviates quasi-transitivity, which is formalized as $\forall xyz (xR_iy\land yR_jz\to xR_jz)$, $pt$ abbreviates pseudo-transitivity, formalized as $\forall xyz (xR_iy\land yR_jz\to xR_1z\land xR_2z)$, $qe$ abbreviates quasi-Euclidicity, formalized $\forall xyz(xR_iy\land xR_jz\to yR_jz)$, and $pe$ abbreviates pseudo-Euclidicity, formalized $\forall xyz(xR_iy\land xR_jz\to yR_1z\land yR_2z)$.
In the above table, $qt,pt,qe,pe$ abbreviate quasi-transitivity, pseudo-transitivity, quasi-Euclidicity and pseudo-Euclidicity, respective, which are formalized by $\forall xyz (xR_iy\land yR_jz\to xR_jz)$, $\forall xyz (xR_iy\land yR_jz\to xR_1z\land xR_2z)$, $\forall xyz(xR_iy\land xR_jz\to yR_jz)$, and $\forall xyz(xR_iy\land xR_jz\to yR_1z\land yR_2z)$, respectively, where $i,j\in\{1,2\}$.

\subsubsection{Serial logic}

Thm.~\ref{thm.k-completeness} shows that ${\bf K^{\boxdot}}$ is the minimal $\boxdot$-logic. We now demonstrate that the same system is also the serial $\boxdot$-logic, that is, ${\bf K^{\boxdot}}$ is also sound and strongly complete with respect to the class of serial bimodal frames. For this, we only need to show that $R^c_1$ and $R^c_2$ are both serial, which though cannot be guaranteed due to the possibility that all formulas of the form $\boxdot\phi$ belongs to some state. Due to the fact that $R^c_1=R^c_2$, we call the points that have neither $R^c_1$- nor $R^c_2$-successors `$R^c$-dead points'.\footnote{Notice that as $R^c_1=R^c_2$, for all $s\in S^c$, $s$ either has both $R^c_1$- and $R^c_2$-successors, or has neither of them.} We handle these points by using a similar strategy to the completeness proof of ${\bf K^{\boxplus}}$ over serial frames (see the remarks before Thm.~\ref{thm.boxplus-serial}). In detail, define $\M^{\bf D}=\lr{S^c,R^{\bf D}_1,R^{\bf D}_2,V^c}$ as $\M^c$ in Def.~\ref{def.cm-boxdot}, except that $R^{\bf D}_i=R^c_i\cup\{(s,s)\mid s\text{ is a }R^c\text{-dead points}\}$. It should be obvious that $\M^{\bf D}$ is serial. Moreover, the truth values of $\mathcal{L}(\boxdot)$-formulas are invariant under the model transformation: for all $s\in S^c$, if $s$ has both $R^c_1$- and $R^c_2$-successors, then it is clear that $\M^c,s\vDash\boxdot\phi$ iff $\M^{\bf D},s\vDash\boxdot\phi$; if $s$ is a $R^c$-dead point, then $\M^c,s\vDash\boxdot\phi$ and $\M^{\bf D},s\vDash\boxdot\phi$, as desired.

The above strategy indicates that $\M^c$ can be transformed into an equivalent serial bimodal model. In the sequel, we will show a stronger result: every bimodal model can be transformed into an equivalent serial bimodal model; more precisely, each bimodal model is a $\boxdot$-morphic image of some serial bimodal model.

Given a bimodal model $\M=\lr{S,R_1,R_2,V}$, each world $s$ in $S$ has four possibilities: $s$ has neither $R_1$-successors nor $R_2$-successors, $s$ has $R_1$-successors but has no $R_2$-successors, $s$ has no $R_1$-successors but has $R_2$-successors, $s$ has both $R_1$-successors and $R_2$-successors. We handle this four different kinds of worlds in different ways, based on the following key observations.
\begin{enumerate}
\item $s$ has neither $R_1$-successors nor $R_2$-successors. In this case, we just add the $R_1$ and $R_2$ arrows from $s$ to itself.
\item $s$ has $R_1$-successors but has no $R_2$-successors. In this case, we first replace $s$ with some of its new copies, such that each copy has only one $R_1$-successor, then add the $R_2$-arrow from each copy to its sole $R_1$-successor.
\item $s$ has no $R_1$-successors but has $R_2$-successors. The method for dealing with this case is similar to that for the second case. We first replace $s$ with some of its new copies, such that each copy has only one $R_2$-successor, then add the $R_1$-arrow from each copy to its sole $R_2$-successor.
\item $s$ has both $R_1$-successors (say $t$) and $R_2$-successors (say $u$). In this case, if for instance, $t$ lies in the first case or the current case, we just keep the point $t$ and the arrow from $s$ to $t$. However, if $t$ lies in other two cases, then we cannot simply do the same thing (otherwise the truth values of formulas may change during the tranformation); instead, we need to replace $t$ with some of its new copies and deal with $t$ in the same way as in the second and third cases.
\end{enumerate}

%$E_1=\{s\in S^c\mid sR^c_2t\text{ for some }t\in S^c \text{ and }sR^c_1u\text{ for no }u\}$

%$E_2=\{s\in S^c\mid sR^c_1t\text{ for some }t\in S^c \text{ and }sR^c_2u\text{ for no }u\}$

%Define $E_1=\{s\in S^c\mid sR^c_1t\text{ for some }t\in S^c\}$ and $E_2=\{s\in S^c\mid sR^c_2t\text{ for some }t\in S^c\}$, and let $\overline{E_1}=S^c\backslash E_1$ and $\overline{E_2}=S^c\backslash E_2$.

%It is not hard to see that $S^c$ can be partitioned into four sections: $\overline{E_1}\cap \overline{E_2}$, $E_1\cap\overline{E_2}$, $\overline{E_1}\cap E_2$ and $E_1\cap E_2$.
%$(\overline{E_1}\cap \overline{E_2})\cup(E_1\cap\overline{E_2})\cup(\overline{E_1}\cap E_2)\cup (E_1\cap E_2)=S^c$

Let $\M=\lr{S,R_1,R_2,V}$. Define $E_1=\{s\in S\mid sR_1t\text{ for some }t\in S\}$ and $E_2=\{s\in S\mid sR_2t\text{ for some }t\in S\}$, and let $\overline{E_1}=S\backslash E_1$ and $\overline{E_2}=S\backslash E_2$.

It is not hard to see that $S$ can be partitioned into four areas: $\overline{E_1}\cap \overline{E_2}$, $E_1\cap\overline{E_2}$, $\overline{E_1}\cap E_2$ and $E_1\cap E_2$.

\begin{definition}\label{def.cm-d}
Given any bimodal model $\M=\lr{S,R_1,R_2,V}$, we construct a bimodal model $\M'=\lr{S',R'_1,R'_2,V'}$, where
\begin{itemize}
\item $S'=(\overline{E_1}\cap \overline{E_2})\cup(E_1\cap E_2)\cup\{(s,t,1)\mid s\in E_1\cap\overline{E_2},sR_1t\}\cup\{(s,t,2)\mid s\in\overline{E_1}\cap E_2,sR_2t\}$
\item $sR'_1t$ iff one of the following conditions holds:
      \begin{enumerate}
      \item $s\in \overline{E_1}\cap \overline{E_2}$ and $s=t$
      \item $s\in E_1\cap E_2$ and $sR_1t$ and $t\in(\overline{E_1}\cap\overline{E_2})\cup(E_1\cap E_2)$
      \item $s\in E_1\cap E_2$ and $t=(t',u,i)\in S'$ and $sR_1t'$, where $i\in\{1,2\}$
      %\item $s\in E_1\cap E_2$ and $t=(t',u,2)\in S'$ and $sR^c_1t'$
      \item $s=(s',t,i)\in S'$ and $t\in(\overline{E_1}\cap\overline{E_2})\cup(E_1\cap E_2)$, where $i\in\{1,2\}$
      \item $s=(s',t',i)\in S'$ and $t=(t',u',j)\in S'$, where $i,j\in\{1,2\}$
      %\item $s=(s',t',1)\in S'$ and $t=(t',u',2)\in S'$
      %\item $s=(s',t,2)\in S'$ and $t\in(\overline{E_1}\cap\overline{E_2})\cup(E_1\cap E_2)$
      %\item $s=(s',t',2)\in S'$ and $t=(t',u',1)\in S'$
      %\item $s=(s',t',2)\in S'$ and $t=(t',u',2)\in S'$
      \end{enumerate}
\weg{\item $sR'_1t$ iff one of the following conditions holds:
      \begin{enumerate}
      \item $s\in \overline{E_1}\cap \overline{E_2}$ and $s=t$
      \item $s\in E_1\cap E_2$ and $sR^c_1t$ and $t\in(\overline{E_1}\cap\overline{E_2})\cup(E_1\cap E_2)$
      \item $s\in E_1\cap E_2$ and $t=(t',u,1)\in S'$ and $sR^c_1t'$
      \item $s\in E_1\cap E_2$ and $t=(t',u,2)\in S'$ and $sR^c_1t'$
      \item $s=(s',t,1)\in S'$ and $t\in(\overline{E_1}\cap\overline{E_2})\cup(E_1\cap E_2)$
      \item $s=(s',t',1)\in S'$ and $t=(t',u',1)\in S'$
      \item $s=(s',t',1)\in S'$ and $t=(t',u',2)\in S'$
      \item $s=(s',t,2)\in S'$ and $t\in(\overline{E_1}\cap\overline{E_2})\cup(E_1\cap E_2)$
      \item $s=(s',t',2)\in S'$ and $t=(t',u',1)\in S'$
      \item $s=(s',t',2)\in S'$ and $t=(t',u',2)\in S'$
      \end{enumerate}}
\item $sR'_2t$ iff one of the following holds:
      \begin{enumerate}
      \item $s\in \overline{E_1}\cap \overline{E_2}$ and $s=t$
      \item $s\in E_1\cap E_2$ and $sR_2t$ and $t\in(\overline{E_1}\cap\overline{E_2})\cup(E_1\cap E_2)$
      \item $s\in E_1\cap E_2$ and $t=(t',u,i)\in S'$ and $sR_2t'$, where $i\in\{1,2\}$
      %\item $s\in E_1\cap E_2$ and $t=(t',u,2)\in S'$ and $sR^c_2t'$
      \item $s=(s'',t,i)\in S'$ and $t\in(\overline{E_1}\cap\overline{E_2})\cup(E_1\cap E_2)$, where $i\in\{1,2\}$
      \item $s=(s'',t'',i)\in S'$ and $t=(t'',u'',j)\in S'$, where $i,j\in\{1,2\}$
      %\item $s=(s'',t'',1)\in S'$ and $t=(t'',u'',2)\in S'$
      %\item $s=(s'',t,2)\in S'$ and $t\in(\overline{E_1}\cap\overline{E_2})\cup(E_1\cap E_2)$
      %\item $s=(s'',t'',2)\in S'$ and $t=(t'',u'',1)\in S'$
      %\item $s=(s'',t'',2)\in S'$ and $t=(t'',u'',2)\in S'$
      \end{enumerate}
\weg{\item $sR'_2t$ iff one of the following holds:
      \begin{enumerate}
      \item $s\in \overline{E_1}\cap \overline{E_2}$ and $s=t$
      \item $s\in E_1\cap E_2$ and $sR^c_2t$ and $t\in(\overline{E_1}\cap\overline{E_2})\cup(E_1\cap E_2)$
      \item $s\in E_1\cap E_2$ and $t=(t',u,1)\in S'$ and $sR^c_2t'$
      \item $s\in E_1\cap E_2$ and $t=(t',u,2)\in S'$ and $sR^c_2t'$
      \item $s=(s'',t,1)\in S'$ and $t\in(\overline{E_1}\cap\overline{E_2})\cup(E_1\cap E_2)$
      \item $s=(s'',t'',1)\in S'$ and $t=(t'',u'',1)\in S'$
      \item $s=(s'',t'',1)\in S'$ and $t=(t'',u'',2)\in S'$
      \item $s=(s'',t,2)\in S'$ and $t\in(\overline{E_1}\cap\overline{E_2})\cup(E_1\cap E_2)$
      \item $s=(s'',t'',2)\in S'$ and $t=(t'',u'',1)\in S'$
      \item $s=(s'',t'',2)\in S'$ and $t=(t'',u'',2)\in S'$
      \end{enumerate}}
\item $V'(p)=\{s\in S'\mid g(s)\in V(p)\}$, where $g$ is a function from $S'$ to $S$ such that $g(s)=s$ for $s\in(\overline{E_1}\cap \overline{E_2})\cup(E_1\cap E_2)$, and $g((s,t,i))=s$ for $(s,t,i)\in S'$ where $i\in\{1,2\}$.

\end{itemize}
\end{definition}

\weg{\begin{definition}
$\M'=\lr{S',R'_1,R'_2,V'}$ is the canonical model, if
\begin{itemize}
\item $S'=\overline{E_1}\cup\overline{E_2}\cup\{(s,t,2)\mid s\in E_1, sR^c_2t\}\cup\{(s,t,1)\mid s\in E_2,sR^c_1t\}$
\item $sR'_1t$ iff one of the following holds:
      \begin{enumerate}
      \item $sR^c_1t$ and $s\in\overline{E_2}$
      \item $s=(s',t,1)\in S'$
      \item $s=(s',t,2)\in S'$
      \end{enumerate}
\item $sR'_2t$ iff one of the following holds:
      \begin{enumerate}
      \item $sR^c_2t$ and $s\in\overline{E_1}$
      \item $s=(s',t,2)\in S'$
      \item $s=(s',t,1)\in S'$
      \end{enumerate}
\end{itemize}
\end{definition}}

It would be constructive to give a concrete example. We choose the following example to cover all conditions in the definitions of the relations $R_1'$ and $R_2'$ (for the sake of simplicity, we leave out the valuations).
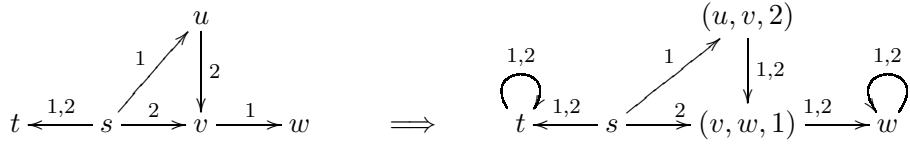
\begin{example}
\weg{$$
\xymatrix{&&&s\ar[dll]_1\ar[drr]^2&&&\\
&t\ar[dl]_1\ar[dr]^1&&&&u\ar[dl]_2\ar[dr]^2&\\
t_1&&t_2&&u_1&&u_2\\}
$$
\small{$$
\xymatrix{&&s\ar[dl]_1\ar[dr]^2&&&&&&s\ar[dl]_1\ar[dll]_1\ar[dr]^2\ar[drr]^2&&\\
&t\ar[dl]_1\ar[d]^1&&u\ar[d]_2\ar[dr]^2&&\Longrightarrow&(t,t_1,1)\ar[d]_{1,2}\ar@/^10pt/[dr]_{1,2}&(t,t_2,1)\ar[d]^{1,2}\ar@/^10pt/[dl]^{1,2}&&(u,u_1,2)\ar[d]_{1,2}\ar@/^10pt/[dr]_{1,2}&(u,u_2,2)\ar[d]^{1,2}\ar@/^10pt/[dl]^{1,2}\\
t_1&t_2&&u_1&u_2&&t_1&t_2&&u_1&u_2\\}
$$}
$$
\xymatrix{&r\ar[d]_1\ar[ddr]^2\ar[dr]^2&&&&&r\ar[d]_1\ar[ddr]^2\ar[ddrr]^2\ar[drr]^{2}&&\\
&s\ar[dl]_1\ar[d]_1\ar[dr]_2&w\ar[d]^1&&&&s\ar[dl]_1\ar[d]_1\ar[dr]_2\ar[drr]^2&&(w,u,1)\ar[d]^{1,2}\ar@/_10pt/[dl]_{1,2}\\
v\ar[r]_2&t\ar[d]_1&u\ar[d]^2\ar[dr]^2&&\Longrightarrow&(v,t,2)\ar[r]_{1,2}&(t,t_1,1)\ar[d]_{1,2}&(u,u_1,2)\ar[d]_{1,2}\ar@/^10pt/[dr]_{1,2}&(u,u_2,2)\ar[d]^{1,2}\ar@/^10pt/[dl]^{1,2}\\
&t_1&u_1&u_2&&&t_1\ar@(dr,dl)^{1,2}&u_1\ar@(dl,dr)_{1,2}&u_2\ar@(dl,dr)_{1,2}\\}
$$}
\[
\xymatrix{&&u\ar[d]^2&&&&&(u,v,2)\ar[d]^{1,2}&\\
t&s\ar[l]_{1,2}\ar[ur]^1\ar[r]^2&v\ar[r]^1&w&\Longrightarrow&t\ar@(ul,ur)^{1,2}&s\ar[l]_{1,2}\ar[ur]^1\ar[r]^2&(v,w,1)\ar[r]^{1,2}&w\ar@(ur,ul)_{1,2}\\}
\]
In the left-hand model $\M$, it is not hard to see that $s\in E_1\cap E_2$, $u\in \overline{E_1}\cap E_2$, $v\in E_1\cap\overline{E_2}$, and $t,w\in \overline{E_1}\cap \overline{E_2}$. Thus in the right-hand model $\M'$, $s,t,w$ are kept unchanged, whereas $u$ and $v$ are replaced by their new copies $(u,v,2)$ (since $uR_2v$), $(v,w,1)$ (since $vR_1w$), respectively.

Now for the arrows in $\M'$, viz. accessibility relations. The $1$- and $2$-arrows from $t$ to itself and from $w$ to itself are obtained from the first conditions of (the definitions of) $R_1'$ and $R_2'$. The $1$- and $2$-arrows from $s$ to $t$ follow from the second conditions of $R_1'$ and $R_2'$. The $1$-arrow from $s$ to $(u,v,2)$ is derived from the third condition of $R_1'$. The $2$-arrow from $s$ to $(v,w,1)$ is deduced from the third condition of $R_2'$. The $1$- and $2$-arrows from $(v,w,1)$ to $w$ are inferred due to the fourth conditions of $R_1'$ and $R_2'$. The $1$- and $2$-arrows from $(u,v,2)$ to $(v,w,1)$ are concluded by the fifth conditions of $R_1'$ and $R_2'$. In this way, we transform the non-serial model $\M$ into the desired serial model $\M'$.
\weg{In the left-hand model $\M$, it is not hard to see that $r,s\in E_1\cap E_2$, $w,t\in E_1\cap \overline{E_2}$, $v,u\in \overline{E_1}\cap E_2$, $t_1,u_1,u_2\in\overline{E_1}\cap \overline{E_2}$. Thus in the right-hand model $\M'$, $r,s,t_1,u_2,u_2$ are kept unchanged, whereas $w$, $v$, $t$, $u$ are replaced by their copies $(w,u,1)$ (since $wR_1u$), $(v,t,2)$ (since $vR_2t$), $(t,t_1,1)$ (since $tR_1t_1$) and $(u,u_1,2)$ (since $uR_2u_1$), $(u,u_2,2)$ (since $uR_2u_2$), respectively.

Now for the arrows in $\M'$, viz. accessibility relations. The $1$-arrow from $r$ to $s$ follows from the second condition of $R_1'$; the $2$-arrows from $r$ to $(u,u_1,2)$ and $(u,u_2,2)$ and $(w,u,1)$ are obtained from the third condition of $R_2'$, and so are the $2$-arrows from $s$ to both $(u,u_1,2)$ and $(u,u_2,2)$; the $1$-arrows from $s$ to both $(v,t,2)$ and $(t,t_1,1)$ are derived from the third conditions of $R_1'$; the $1$-arrow and $2$-arrow from $(w,u,1)$ to both $(u,u_1,2)$ and $(u,u_2,2)$ obtain via the fifth conditions of $R_1'$ and $R_2'$, and so do the $1$-arrow and $2$-arrow from $(v,t,2)$ to $(t,t_1,1)$; the $1$-arrow and $2$-arrow from $(t,t_1,1)$ to $t_1$ are inferred due to the fourth conditions of $R_1'$ and $R_2'$, and so are the $1$-arrows and $2$-arrows from both $(u,u_1,2)$ and $(u,u_2,2)$ to both $u_1$ and $u_2$; the $1$-arrow and $2$-arrow from $t_1$ to itself are concluded because of the first conditions of $R_1'$ and $R_2'$.}
\end{example}

The following proposition states that $\M'$ constructed via Def.~\ref{def.cm-d} is indeed serial.

\begin{proposition}\label{prop.serial}
$\M'$ is serial.
\end{proposition}

\begin{proof}
%It suffices to show that $R_1'$ and $R_2'$ are both serial.
Let $s\in S'$ be arbitrary. We need to show that there are $x,y\in S'$ such that $sR'_1x$ and $sR'_2y$.

According to the definition of $S'$, we distinguish the following cases.
\begin{enumerate}
\item $s\in \overline{E_1}\cap\overline{E_2}$. Then by the first conditions of the definitions of $R_1'$ and $R'_2$, $s$ is the desired $x$ and $y$.
\item $s\in E_1\cap E_2$. Then $sR_1t$ for some $t\in S$. We consider all possibilities of $t$ as follows.
      \begin{enumerate}
      \item $t\in (\overline{E_1}\cap \overline{E_2})\cup(E_1\cap E_2)$. According to the second condition of the definition of $R'_1$, we have $sR'_1t$, and thus $t$ is the desired $x$.
      %\item $t\in E_1\cap \overline{E_2}$. Then $tR^c_1u$ for some $u\in S^c$. Then $(t,u,1)\in S'$. According to the third condition of the definition of $R'_1$, we infer $sR'_1(t,u,1)$, thus $(t,u,1)$ is the desired $x$.
      \item $t\in (E_1\cap \overline{E_2})\cup(\overline{E_1}\cap E_2)$. Then $tR_iu$ for some $u\in S^c$, where the value of $i$ depends on $t$: if $t\in E_1\cap \overline{E_2}$, then $i=1$; otherwise $i=2$. Then $(t,u,i)\in S'$. According to the third condition of the definition of $R'_1$, we infer $sR'_1(t,u,i)$, thus $(t,u,i)$ is the desired $x$.
      %\item $t\in \overline{E_1}\cap E_2$. Then $tR^c_2u$ for some $u\in S^c$. Then $(t,u,2)\in S'$. According to the fourth condition of the definition of $R'_1$, we derive $sR'_1(t,u,2)$, thus $(t,u,2)$ is the desired $x$.
      \end{enumerate}
      We have also $sR_2u$ for some $u\in S$. With a similar argument, we can obtain $sR'_2y$ for some $y\in S'$.
\item $s=(s',t,i)\in S'$ where $i\in\{1,2\}$. Then $s'\in E_i\cap \overline{E_j}$ and $s'R_it$, where $j\in\{1,2\}$ and $j\neq i$. Again, since $t\in S$, we consider all possibilities of $t$ as follows.
    \begin{enumerate}
    \item $t\in (\overline{E_1}\cap \overline{E_2})\cup(E_1\cap E_2)$. According to the fourth conditions of the definitions of $R'_1$ and $R'_2$, we get $sR'_1t$ and $sR'_2t$, and thus $t$ is the desired $x$ and $y$.
    \item $t\in (E_1\cap \overline{E_2})\cup(\overline{E_1}\cap E_2)$. Then $tR_ku$ for some $u\in S$, where the value of $k$ depends on $t$: if $t\in E_1\cap \overline{E_2}$, then $k=1$; otherwise $k=2$. Then $(t,u,k)\in S'$. According to the fifth conditions of the definitions of $R'_1$ and $R'_2$, we have $sR'_1(t,u,k)$ and also $sR'_2(t,u,k)$, and thus $(t,u,k)$ is the desired $x$ and $y$.
    %\item $t\in \overline{E_1}\cap E_2$. Then $tR^c_2u$ for some $u\in S^c$. Then $(t,u,2)\in S'$. According to the seventh condition of the definition of $R'_1$, we have $sR'_1(t,u,2)$, and thus $(t,u,2)$ is the desired $x$.
    \end{enumerate}
\end{enumerate}
We have thus shown that in all cases, there always exist $x,y\in S'$ such that $sR'_1x$ and $sR'_2y$, as desired.
\end{proof}

The proposition below indicates that $g$ satisfies the condition (Forth) of a $\boxdot$-morphism.

\begin{proposition}\label{prop.forth} If $sR'_1t$ and $sR'_2u$ and $g(t)\neq g(u)$, then $g(s)R_1g(t)$ and $g(s)R_2g(u)$.
\end{proposition}

\begin{proof}
Suppose that $sR'_1t$ and $sR'_2u$ and $g(t)\neq g(u)$, thus $t\neq u$. Since $s\in S'$, we consider the following cases.
\begin{enumerate}
\item $s\in\overline{E_1}\cap \overline{E_2}$. According to the first condition of the definition of $R'_1$ and $R'_2$, we would have $s=t$ and $s=u$, which implies that $t=u$. Contradiction.
\item $s\in E_1\cap E_2$. Then $g(s)=s$. Since $sR'_1t$ and $sR'_2u$, according to the second and third conditions of the definitions of $R_1'$ and $R'_2$, we consider four subcases.
    \begin{enumerate}
    \item $sR_1t$ and $t\in(\overline{E_1}\cap\overline{E_2})\cup(E_1\cap E_2)$ and $sR_2u$ and $u\in(\overline{E_1}\cap\overline{E_2})\cup(E_1\cap E_2)$. In this case, we have $g(t)=t$ and $g(u)=u$, and therefore $g(s)R_1g(t)$ and $g(s)R_2g(u)$.
    \item $sR_1t$ and $t\in(\overline{E_1}\cap\overline{E_2})\cup(E_1\cap E_2)$ and $u=(u',y,i)\in S'$ and $sR_2u'$, where $i\in\{1,2\}$. In this case, $g(t)=t$ and $g(u)=u'$, and therefore $g(s)R_1g(t)$ and $g(s)R_2g(u)$.
    \item $t=(t',x,i)\in S'$ and $sR_1t'$, where $i\in\{1,2\}$ and $sR_2u$ and $u\in(\overline{E_1}\cap\overline{E_2})\cup(E_1\cap E_2)$. In this case, $g(t)=t'$ and $g(u)=u$, and then $g(s)R_1g(t)$ and $g(s)R_2g(u)$.
    \item $t=(t',x,i)\in S'$ and $sR_1t'$ and $u=(u',y,j)\in S'$ and $sR_2u'$, where $i,j\in\{1,2\}$. In this case, we have $g(t)=t'$ and $g(u)=u'$, and therefore $g(s)R_1g(t)$ and $g(s)R_2g(u)$.
    \end{enumerate}
\item $s$ is of the form $(x,y,i)\in S'$, where $i\in\{1,2\}$. Since $sR'_1t$ and $sR'_2u$, according to the fourth and fifth conditions of the definitions of $R'_1$ and $R'_2$, we consider four subcases.
    \begin{enumerate}
    \item $s=(s',t,i)\in S'$ and $s=(s'',u,j)\in S'$ and $t,u\in(\overline{E_1}\cap\overline{E_2})\cup(E_1\cap E_2)$, where $i,j\in\{1,2\}$. In this case, we would have $t=u$: a contradiction.
    \item $s=(s',t,i)\in S'$ and $t\in(\overline{E_1}\cap\overline{E_2})\cup(E_1\cap E_2)$ and $u=(t,y,j)\in S'$, where $i,j\in\{1,2\}$. In this case, $t\in (E_1\cap\overline{E_2})\cup(\overline{E_1}\cap E_2)$: a contradiction.
    \item $s=(s'',u,i)\in S'$ and $u\in(\overline{E_1}\cap\overline{E_2})\cup(E_1\cap E_2)$ and $t=(u,x,j)\in S'$, where $i,j\in\{1,2\}$. In this case, $u\in (E_1\cap\overline{E_2})\cup(\overline{E_1}\cap E_2)$: a contradiction.
    \item $s=(s',t',i)\in S'$ and $t=(t',x,j)\in S'$ and $u=(t',y,k)\in S'$, where $i,j,k\in\{1,2\}$. In this case, we would have $g(t)=g(u)=t'$: a contradiction.
    \end{enumerate}
\end{enumerate}
\end{proof}

It is worth remarking that the precondition `$g(t)\neq g(u)$' in the statement of the above proposition cannot be weakened to `$t\neq u$'. For instance, in $\M$, $s'R_1t'$ and $t'R_1x$ and $t'R_1y$ and $x\neq y$ but $s'$ and $t'$ both have no $R_2$-successors. According to the fifth conditions of our definitions of $R'_1$ and $R'_2$, in $\M'$, $(s',t',1)R_1'(t',x,1)$ and $(s',t',1)R_2'(t',y,1)$ and $(t',x,1)\neq (t',y,1)$. However, $g(s',t',1)=s'$, which implies that $g(s',t',1)$ has no $R_2$-successors, thus we have no $g(s',t',1)R_2g(t',y,1)$.

%\begin{proposition}\label{prop.back}[Back] If $g(s)R_1t'$ and $g(s)R_2u'$\weg{ and $t'\neq u'$}, then there are $t$ and $u$ in $S'$ such that $sR'_1t$ and $sR'_2u$ and $g(t)=t'$ and $g(u)=u'$.
%\end{proposition}
\medskip

The following result states that $g$ also satisfies the condition (Back) of a $\boxdot$-morphism.
\begin{proposition}\label{prop.back} If $g(s)R_1t'$ and $g(s)R_2u'$ and $t'\neq u'$, then there are $t$ and $u$ in $S'$ such that $sR'_1t$ and $sR'_2u$ and $g(t)=t'$ and $g(u)=u'$.
\end{proposition}

\begin{proof}
We show a stronger result:

($\ast$)~~~If $g(s)R_1t'$ and $g(s)R_2u'$, then there are $t$ and $u$ in $S'$ such that $sR'_1t$ and $sR'_2u$ and $g(t)=t'$ and $g(u)=u'$.

Assume that $g(s)R_1t'$ and $g(s)R_2u'$. It is easy to see that $g(s)\in E_1\cap E_2$. Then we must have $g(s)=s$: otherwise, by the definition of $g$, $g(s)=s'$ and $s=(s',x,i)\in S'$ where $i\in\{1,2\}$, then $s'\in E_1\cap E_2$ and {\em either} $s'\in E_1\cap \overline{E_2}$ {\em or} $s'\in\overline{E_1}\cap E_2$, which is impossible. Thus $sR_1t'$ and $sR_2u'$. Since $t'\in S$, we have the following cases.
\begin{itemize}
\item $t'\in(\overline{E_1}\cap\overline{E_2})\cup(E_1\cap E_2)$. Then by the second conditions of $R'_1$, it follows that $sR'_1t'$; by the definition of $g$, $g(t')=t'$. Therefore, $t'$ is the desired $t$.
\item $t'\in E_i\cap \overline{E_j}$, where $i,j\in\{1,2\}$ and $i\neq j$.
    In this case, $t'R_ix$ for some $x$, then $(t',x,i)\in S'$. By the third condition of the definition of $R'_1$, $sR'_1(t',x,i)$; by the definition of $g$, $g(t',x,i)=t'$. Therefore, $(t',x,i)$ is the desired $t$.
\end{itemize}
We have thus shown that there exists $t\in S'$ such that $sR'_1t$ and $g(t)=t'$.

Similarly, from $u'\in S$ and $sR_2u'$, we can show that there exists $u\in S'$ such that $sR'_2u$ and $g(u)=u'$, as desired.
%We have now shown that in all cases, there are $t$ and $u$ in $S'$ such that $sR'_1t$ and $sR'_2u$ and $g(t)=t'$ and $g(u)=u'$, as desired.
\weg{(2) Assume that $g(s)R^c_1t'$ and $g(s)R^c_2u'$. It is easy to see that $g(s)\in E_1\cap E_2$. Then we must have $g(s)=s$: otherwise, by the definition of $g$, $g(s)=s'$ and $s=(s',x,i)\in S'$, then $s'\in E_1\cap E_2$ and either $s'\in E_1\cap \overline{E_2}$ or $s'\in\overline{E_1}\cap E_2$, which is impossible. Thus $sR^c_1t'$ and $sR^c_2u'$. By the second condition of the definitions of $R'_1$ and $R'_2$, we have $sR'_1t'$ and $sR'_2u'$.}
\end{proof}

We have now shown that $g$ is a $\boxdot$-morphism from $\M'$ to $\M$. Then by Prop.~\ref{prop.boxdot-modelproperty}, we immediately have

\begin{lemma}\label{lem.equivalent}
For all $s\in S'$, for all $\phi\in\mathcal{L}(\boxdot)$, we have
$$\M',s\vDash\phi\iff\M,g(s)\vDash\phi.$$
\end{lemma}

\weg{\begin{proof}
By induction on $\phi$. The nontrivial case is $\boxdot\phi$.

Suppose that $\M',s\nvDash\boxdot\phi$. Then there are $t,u\in S'$ such that $sR'_1t$ and $sR'_2u$ and it is {\em not} the case that $\M',t\vDash\phi\iff \M',u\vDash\phi$. By induction hypothesis, it is {\em not} the case that $\M,g(t)\vDash\phi\iff \M,g(u)\vDash\phi$, and thus $g(t)\neq g(u)$. By Prop.~\ref{prop.forth}, $g(s)R_1g(t)$ and $g(s)R_2g(u)$, and hence $\M,g(s)\nvDash\boxdot\phi$.

Conversely, assume that $\M,g(s)\nvDash\boxdot\phi$. Then there are $t',u'\in S$ such that $g(s)R_1t'$ and $g(s)R_2u'$ and it is {\em not} the case that $\M,t'\vDash\phi\iff \M,u'\vDash\phi$. By Prop.~\ref{prop.back}, there are $t,u\in S'$ such that $sR_1't$ and $sR_2'u$ and $g(t)=t'$ and $g(u)=u'$. By induction hypothesis, it is {\em not} the case that $\M,t\vDash\phi\iff \M,u\vDash\phi$, and therefore $\M',s\nvDash\boxdot\phi$, as desired.
\end{proof}}

To show the completeness, we also need the following result.
\begin{lemma}\label{lem.surjective}
$g$ is surjective.
\end{lemma}

\begin{proof}
Suppose that $s\in S$, to find a $x\in S'$ such that $g(x)=s$. We consider two cases.
\begin{itemize}
\item $s\in(\overline{E_1}\cap \overline{E_2})\cup (E_1\cap E_2)$. According to the definition of $g$, we have $g(s)=s$; clearly, $s\in S'$.
\item $s\in E_i\cap \overline{E_j}$, where $i,j\in\{1,2\}$ and $i\neq j$. Then $sR_it$ for some $t$. It follows that $(s,t,i)\in S'$. By the definition of $g$, we have $g(s,t,i)=s$.
\end{itemize}
\end{proof}

\begin{theorem}
${\bf K^{\boxdot}}$ is sound and strongly complete with respect to the class of serial bimodal frames.
\end{theorem}

\begin{proof}
Let $\Gamma$ be a consistent set. By Thm.~\ref{thm.k-completeness}, $\Gamma$ is satisfiable in a bimodal model, say $(\M,s)$. We then construct $\M'$ from $\M$ as in Def.~\ref{def.cm-d}. By Lemma~\ref{lem.surjective}, there exists $x\in S'$ such that $g(x)=s$, and thus $\M,g(x)\vDash\Gamma$. Then by Lemma~\ref{lem.equivalent}, $\M',x\vDash\Gamma$. We also know that $\M'$ is serial by Prop.~\ref{prop.serial}. Therefore, $\Gamma$ is satisfiable in a serial bimodal model, as desired.
\end{proof}

\weg{\begin{proof}
Define $\M^c$ as in Def.~\ref{def.cm-boxdot}. We reflexivize all $R_1^c$-dead ends and $R_2^c$-dead ends; formally, let $R_1^{\bf D}=\{(u,u)\mid u\text{ is a dead end w.r.t. }R^c_1\}$ and $R_2^{\bf D}=\{(u,u)\mid u\text{ is a dead end w.r.t. }R_2^c\}$ and then define $\M^{\bf D}=\lr{S^c,R_1^{\bf D},R_2^{\bf D},V^c}$. It is now clear that $\M^{\bf D}$ is serial.

Moreover, the truth values of $\mathcal{L}(\boxdot)$-formulas are preserved during the transformation, that is, for all $s\in S^c$ and for all $\phi\in\mathcal{L}(\boxdot)$, we have $\M^c,s\vDash\phi$ iff $\M^{\bf D},s\vDash\phi$. The proof goes with induction on $\phi$, in which the nontrivial case is $\boxdot\phi$. If $s$ is not a $R^c_1$-dead end nor $R^c_2$-dead end, then the claim is clear; otherwise, we have $\M^c,s\vDash\boxdot\phi$ and also
\end{proof}}

\subsubsection{Reflexive logic}

In this section, we show that ${\bf KT^\boxdot}$ is sound and strongly complete with respect to the class of reflexive bimodal frames. As we will see, ${\bf KT^\boxdot}$ is also sound and strongly complete with respect to the class of bimodal frames $\lr{S,R_1,R_2}$ where either $R_1$ or $R_2$ is reflexive.

%$\boxdot\text{T}$: $\phi\to[\boxdot\phi\to(\boxdot(\phi\to\psi)\to\boxdot\psi)]$

\begin{proposition}\label{prop.validity-t}
$\boxdot\text{T}$ is valid on the class of reflexive bimodal frames.
\end{proposition}

\begin{proof}
Let $\M=\lr{S,R_1,R_2,V}$ be an arbitrary reflexive bimodal model and $s\in S$. Suppose that $\M,s\vDash\phi\land\boxdot\phi\land\boxdot(\phi\to\psi)$, to show that $s\vDash\boxdot\psi$. Since $s\in R_1(s)\cap R_2(s)$, by Coro.~\ref{coro.equiv}, from $\M,s\vDash\boxdot\phi\land\boxdot(\phi\to\psi)$ it follows that $s\vDash\Delta_1\phi\land\Delta_2\phi\land\Delta_1(\phi\to\psi)\land\Delta_2(\phi\to\psi)$. By the obtained result,\footnote{That is, $\phi\land\Delta\phi\land\Delta(\phi\to\psi)\to\Delta\psi$ is valid over the class of reflexive frames $\lr{S,R}$, see e.g.~\cite{Fanetal:2015}.} we can show that $s\vDash\Delta_1\psi\land\Delta_2\psi$. Now using Coro.~\ref{coro.equiv}, we conclude that $\M,s\vDash\boxdot\psi$.
\end{proof}

As one may easily verify, the above statement still holds if the class of reflexive bimodal frames is enlarged to the class of bimodal frames where at least one accessibility relation is reflexive, that is, $\boxdot\text{T}$ is valid over bimodal frames $\lr{S,R_1,R_2}$ where $R_1$ or $R_2$ is reflexive.

\weg{\begin{proposition}
$\boxdot\text{T}$ is valid over bimodal frames where $R_1$ or $R_2$ is reflexive.
\end{proposition}

\begin{proof}
Suppose not, that is, for some model $\M=\lr{S,R_1,R_2,V}$ where $R_1$ or $R_2$ is reflexive such that $\M,s\vDash\phi\land\boxdot\phi\land\boxdot(\phi\to\psi)$ but $s\nvDash\boxdot\psi$. Then there are $t,u\in S$ such that $sR_1t$ and $sR_2u$ and {\em either} ($t\vDash\psi$ and $u\nvDash\psi$) {\em or} ($t\nvDash\psi$ and $u\vDash\psi$).  In the first case, we have $t\vDash\phi\to\psi$, from which by $s\vDash\boxdot(\phi\to\psi)$ it follows that $u\vDash\phi\to\psi$ and thus $u\nvDash\phi$. This then implies that $t\nvDash\phi$ due to the fact that $s\vDash\boxdot\phi$. If $R_1$ is reflexive, then $sR_1s$, and by $s\vDash\boxdot\phi$ and $u\nvDash\phi$ and $sR_2u$, we should have $s\nvDash\phi$: a contradiction; if $R_2$ is reflexive, then $sR_2s$, and by $s\vDash\boxdot\phi$ and $t\nvDash\phi$ and $sR_1t$, we should have also $s\nvDash\phi$: again a contradiction.

In the second case,
\end{proof}}

\begin{definition}\label{def.cm-t}
Define $\M^c$ w.r.t. ${\bf KT^\boxdot}$ as in Def.~\ref{def.cm-boxdot}. We say $\M^r=\lr{S^c,R_1^r,R^r_2,V^c}$ is the {\em reflexive closure} of $\M^c$, if for all $i\in\{1,2\}$, $R^r_i$ is the reflexive closure of $R^c_i$; in symbol, $R_i^r=R^c_i\cup\{(s,s)\mid s\in S\}$ for $i\in\{1,2\}$.
\end{definition}

It is clear that $\M^r$ is a reflexive bimodal model.

\begin{lemma}[Truth Lemma for ${\bf KT^\boxdot}$]\label{lem.truthlemma-t}
For all $s\in S^c$, for all $\phi\in\mathcal{L}(\boxdot)$, we have
$$\phi\in s\text{ iff }\M^r,s\vDash\phi.$$
\end{lemma}

\begin{proof}
By induction on $\phi$. We only consider the nontrivial case $\boxdot\phi$, that is to show, $\boxdot\phi\in s$ iff $\M^r,s\vDash\boxdot\phi$.

`If': straightforward by Lemma~\ref{lem.truth-lemma-k} and the fact that $R^c_i\subseteq R^r_i$ for $i\in\{1,2\}$.

`Only if': Suppose, for a contradiction, that $\boxdot\phi\in s$ but $\M^r,s\nvDash\boxdot\phi$. By induction hypothesis, there exist $t,u\in S^c$ such that $sR^r_1t$ and $sR^r_2u$ and $(\phi\in t\not\iff\phi\in u)$. W.l.o.g. we may assume that $\phi\in t$ but $\phi\notin u$. If $s\neq t$ and $s\neq u$, then $sR^c_1t$ and $sR^c_2u$, and thus the proof continues as in the corresponding part in Lemma~\ref{lem.truth-lemma-k}, and finally we can arrive at a contradiction. If $s=t$ or $s=u$, w.l.o.g. we assume that $s=t$, and thus $s\neq u$ (as $t\neq u$), hence $sR^c_2u$.

Since $s=t$ and $\phi\in t$, we have $\phi\in s$. Because $sR^c_2u$, there is a $\chi$ such that $\neg\boxdot\chi\in s$ and $(\dag)$: for all $\psi$, if $\boxdot\psi\land\boxdot(\chi\to\psi)\in s$, then $\psi\in u$. By supposition $\boxdot\phi\in s$ and the fact that $\phi\notin u$, we derive that $\boxdot(\chi\to\phi)\notin s$, that is, $\boxdot(\phi\vee\neg\chi)\notin s$. Moreover, by axiom $\boxdot\text{T}$, $\vdash\phi\to[\boxdot\phi\to(\boxdot(\phi\to\chi)\to\boxdot\chi)]$, then as $\phi\land\boxdot\phi\land\neg\boxdot\chi\in s$, $\boxdot(\neg\phi\vee\chi)\notin s$. Now by axiom $\boxdot\text{DIS}$, it follows that $\boxdot\phi\notin s$: a contradiction again.
\end{proof}

It is natural to ask if the above claim can be generalized to any bimodal model, that is, if every bimodal model has an equivalent reflexive closure. The answer is negative. For example, the following are a bimodal model and its reflexive closure, but one may check that $\M,w\vDash\boxdot p$ whereas $\M^r,w\nvDash\boxdot p$.
\[
\xymatrix{\M&w:p\ar[rr]|1&&v:\neg p&&\M^r&w:p\ar@(ur,ul)|{1,2}\ar[rr]|1&&v:\neg p\ar@(ul,ur)|{1,2}\\}
\]

With the soundness of ${\bf K^{\boxdot}}$ (Thm.~), Prop.~\ref{prop.validity-t} and its subsequent remark, Lindenbaum's Lemma, and Lemma~\ref{lem.truthlemma-t} in hand, the following result now follows straightforwardly.
\begin{theorem}
${\bf KT^{\boxdot}}$ is sound and strongly complete with respect to the class of reflexive bimodal frames, and also with respect to the class of bimodal frames $\lr{S,R_1,R_2}$ where either $R_1$ or $R_2$ is reflexive.
\end{theorem}

\subsubsection{Symmetric logic}

This part deals with the soundness and strong completeness of ${\bf KB^\boxdot}$ over the class of symmetric bimodal frames. For the soundness, it suffices to show the validity of $\boxdot\text{B}$. Recall that $\boxdot\text{B}$ denotes $\phi\to\boxdot((\boxdot\phi\land\boxdot(\phi\to\psi)\land\neg\boxdot\psi)\to\chi)$.

%$\boxdot\text{B}$: $\phi\to\boxdot((\boxdot\phi\land\boxdot(\phi\to\psi)\land\neg\boxdot\psi)\to\chi)$

\begin{proposition}
$\boxdot\text{B}$ is valid over the class of symmetric bimodal frames.
\end{proposition}

\begin{proof}
Let $\M=\lr{S,R_1,R_2,V}$ be a symmetric bimodal model and $s\in S$. Suppose, for a contradiction, that $\M,s\vDash\phi$ but $\M,s\nvDash\boxdot((\boxdot\phi\land\boxdot(\phi\to\psi)\land\neg\boxdot\psi)\to\chi)$. Then there exist $t,u$ such that $sR_1t$ and $sR_2u$ such that it is {\em not} the case that ($t\vDash(\boxdot\phi\land\boxdot(\phi\to\psi)\land\neg\boxdot\psi)\to\chi$ iff $u\vDash(\boxdot\phi\land\boxdot(\phi\to\psi)\land\neg\boxdot\psi)\to\chi$). W.l.o.g. we may assume that $u\nvDash(\boxdot\phi\land\boxdot(\phi\to\psi)\land\neg\boxdot\psi)\to\chi$, i.e. $u\vDash(\boxdot\phi\land\boxdot(\phi\to\psi)\land\neg\boxdot\psi)\land\neg\chi$.\footnote{The other case that $t\nvDash(\boxdot\phi\land\boxdot(\phi\to\psi)\land\neg\boxdot\psi)\to\chi$ can be shown similarly, by using the symmetry of $R_1$ instead.}

By $u\vDash\neg\boxdot\psi$, there are $v,w$ such that $uR_1v$ and $uR_2w$ and ($v\vDash\psi\not\iff w\vDash\psi$). Since $sR_2u$ and $R_2$ is symmetric, we have $uR_2s$. Since $s\vDash\phi$ and $u\vDash\boxdot\phi$ and $uR_1v$, it follows that $v\vDash\phi$, and thus $w\vDash\phi$. Together with ($v\vDash\psi\not\iff w\vDash\psi$), this implies that $u\nvDash\boxdot(\phi\to\psi)$: a contradiction.
\end{proof}

For the strong completeness, we adopt the following strategy: first show that ${\bf KB^\boxdot}$ is strongly complete with respect to the class of quasi-symmetric bimodal frames, then demonstrate that every quasi-symmetric bimodal model is a $\boxdot$-morphic image of some symmetric bimodal model.

We first note that $\phi\to\boxdot(\boxdot\phi\land\boxdot(\phi\to\psi)\land\neg\boxdot\psi)$, denoted by $\text{w}\boxdot\text{B}$, is derivable in ${\bf KB^\boxdot}$.

\begin{proposition}\label{prop.almost-sym}
Let $i\in\{1,2\}$ and $s,t\in S^c$ such that $\neg\boxdot\chi\in t$. If $sR_i^ct$, then $tR^c_is$.
\end{proposition}

\begin{proof}
Suppose, for a contradiction, that $\neg\boxdot\chi\in t$ and $sR_i^ct$ where $i\in\{1,2\}$ but it is {\em not} the case that $tR^c_is$. Then from $sR_i^ct$, it follows that there exists $\psi$ such that $\neg\boxdot\psi\in s$ and $(\star)$: for all $\delta$, if $\boxdot\delta\land\boxdot(\psi\to\delta)\in s$, then $\delta\in t$. From $\neg\boxdot\chi\in t$ and $\sim tR^c_is$, it follows that there exists $\phi$ such that $\boxdot\phi\land\boxdot(\chi\to\phi)\in t$ but $\phi\notin s$ (that is, $\neg\phi\in s$). By axiom $\boxdot\text{B}$, $\boxdot((\boxdot\neg\phi\land\boxdot(\neg\phi\to\neg\chi)\land\neg\boxdot\neg\chi)\to\neg\psi)\in s$; by $\text{w}\boxdot\text{B}$, $\boxdot(\boxdot\neg\phi\land\boxdot(\neg\phi\to\neg\chi)\land\neg\boxdot\neg\chi)\in s$. Using axioms $\boxdot\text{Equ}$ and $\text{PC}$ and the rule $\text{RE}\boxdot$, we can show that $\boxdot\neg(\boxdot\phi\land\boxdot(\chi\to\phi)\land\neg\boxdot\chi)\in s$ and $\boxdot(\psi\to \neg(\boxdot\phi\land\boxdot(\chi\to\phi)\land\neg\boxdot\chi))\in s$. Then by $(\star)$, $\neg(\boxdot\phi\land\boxdot(\chi\to\phi)\land\neg\boxdot\chi))\in t$: a contradiction.
\end{proof}

\begin{proposition}\label{prop.equiv}
Let $s\in S^c$. Then the following conditions are equivalent:

(1) $\neg\boxdot\chi\in s$ for some $\chi$;

(2) $sR^c_1t$ for some $t$;

(3) $sR^c_2u$ for some $u$.
\end{proposition}

\begin{proof}
%We show that $(1)\Rightarrow(2)\&(3)$ and $(2)\Rightarrow(1)$ and $(3)\Rightarrow(1)$.
$(1)\Rightarrow(2)\&(3)$ can be obtained from item 4 of Prop.~\ref{prop.boxdot-property}, whereas $(2)\Rightarrow(1)$ and $(3)\Rightarrow(1)$ follows from the definitions of $R^c_1$ and $R^c_2$.
\end{proof}

As a corollary of Prop.~\ref{prop.almost-sym} and Prop.~\ref{prop.equiv}, we obtain the following result.
\begin{corollary}\label{coro.almost-sym}
Let $i,j\in\{1,2\}$ and $s,t\in S^c$ such that $tR^c_ju$ for some $u\in S^c$. If $sR^c_it$, then $tR^c_is$.\footnote{In fact, we can get an alternative result: let $i,j\in\{1,2\}$ and $s,t\in S^c$ such that $\neg\boxdot\chi\in t$. If $sR_i^ct$, then $tR^c_js$. This is due to the fact that $R^c_1=R^c_2$. But for our purpose of showing that every quasi-symmetric bimodal model is a $\boxdot$-morphic image of some symmetric bimodal model, we do not need the stronger correspondent (we say `stronger' because in any quasi-symmetric bimodal model $\M=\lr{S,R_1,R_2,V}$, we do not have $R_1=R_2$ in general).}%has a symmetric bimodal model as its $\boxdot$-morphic pre-image.}
\end{corollary}

Given a bimodal model $\M=\lr{S,R_1,R_2,V}$, $\M$ is {\em quasi-symmetric}, if for $i,j\in\{1,2\}$, for all $s,t\in S$ with $tR_ju$ for some $u\in S$, $sR_it$ implies $tR_is$. Intuitively, for any point in a quasi-symmetric model, if it has a successor with respect to some index, then there is a converse arrow with respect to an index from that point to its predecessor (if any). With the notion in mind, it follows from Lemma~\ref{lem.truth-lemma-k} and Coro.~\ref{coro.almost-sym} that
\begin{theorem}\label{thm.comp-quasi-symmetric}
${\bf KB^\boxdot}$ is strongly complete with respect to the class of quasi-symmetric bimodal frames.
\end{theorem}

\medskip
%\subsection{Completeness of ${\bf KB^\boxdot}$}

Given a quasi-symmetric bimodal model $\M=\lr{S,R_1,R_2,V}$, to build a desired symmetric bimodal model, we need only handle those states in $\M$ that have {\em either} $R_1$-predecessors {\em or} $R_2$-predecessors but have neither $R_1$-successors nor $R_2$-successors. We collect as $T_1$ those states in $\M$ that have $R_1$-predecessors but have neither $R_1$-successors nor $R_2$-successors, and collect as $T_2$ those states in $\M$ that have $R_2$-predecessors but have neither $R_1$-successors nor $R_2$-successors. In symbol,
$$T_1=\{t\in S\mid sR_1t\text{ for some }s\in S, \text{ and }tR_1u\text{ for no }u\in S, \text{ and }tR_2v\text{ for no }v\in S\},$$
$$T_2=\{t\in S\mid sR_2t\text{ for some }s\in S, \text{ and }tR_1u\text{ for no }u\in S, \text{ and }tR_2v\text{ for no }v\in S\},$$
and we also define $\overline{T_1}=S\backslash T_1$ and $\overline{T_2}=S\backslash T_2$.

%$D_i=\{t\in S\mid sR_it\text{ for some }s\in S, \text{ and }tR_ju\text{ for no }u\in S\}$, where $i,j\in\{1,2\}$. %Intuitively, $D_i$ consists of all states in $\M$ that has a $R_i$-predecessor but has neither $R_1$-successors nor $R_2$-successors.

\begin{definition}\label{def.symmetric-model}
Given any quasi-symmetric bimodal model $\M=\lr{S,R_1,R_2,V}$, we define a bimodal model $\M^+=\lr{S^+,R^+_1,R^+_2,V^+}$ in which
\begin{itemize}
\item $S^+=\overline{T_1}\cup\overline{T_2}\cup\{(s,t,1)\mid t\in T_1\text{ and }sR_1t\}\cup\{(s,t,2)\mid t\in T_2\text{ and }sR_2t\}$.
\item $sR^+_1t$ iff one of the following conditions holds:
\begin{enumerate}
\item[(i)] $t\in\overline{T_1}$ and $sR_1t$
%\item[(ii)] $s\in\overline{T_2}$ and $tR_2s$
\item[(ii)] $t=(s,t',1)\in S^+$
\item[(iii)] $s=(t,s',1)\in S^+$
%\item[(iv)] $s=(t,s',2)\in S^+$
\end{enumerate}
\weg{\begin{enumerate}
\item[(i)] $t\in\overline{T_1}$ and $sR_1t$
\item[(ii)] $s\in\overline{T_2}$ and $tR_2s$
\item[(iii)] $t=(s,t',1)\in S^+$
\item[(iv)] $s=(t,s',2)\in S^+$
\end{enumerate}}
\item $sR^+_2t$ iff one of the following conditions holds:
\begin{enumerate}
\item[(i)] $t\in\overline{T_2}$ and $sR_2t$
%\item[(ii)] $s\in\overline{T_1}$ and $tR_1s$
\item[(ii)] $t=(s,t'',2)\in S^+$
\item[(iii)] $s=(t,s'',2)\in S^+$
%\item[(iv)] $s=(t,s',1)\in S^+$
\end{enumerate}
\weg{\begin{enumerate}
\item[(i)] $t\in\overline{T_2}$ and $sR_2t$
\item[(ii)] $s\in\overline{T_1}$ and $tR_1s$
\item[(iii)] $t=(s,t',2)\in S^+$
\item[(iv)] $s=(t,s',1)\in S^+$
\end{enumerate}}
\item $V^+(p)=\{s\in S^+\mid h(s)\in V(p)\}$, where $h$ is a function from $S^+$ to $S$ such that $h(s)=s$ for $s\in \overline{T_1}\cup\overline{T_2}$, and $h((s,t,i))=t$ for $(s,t,i)\in S^+$, where $i\in\{1,2\}$.
\end{itemize}
\end{definition}

Note that $\M^+$ in~\cite[Def.~5.9]{Fanetal:2014} is a special case of $\M^+$ here when $R_1^+=R^+_2=R^+$, since $\M^c$ therein is an almost symmetric model and thus a quasi-symmetric model, and the condition that $p\in h(s)$ is equivalent to the condition that $h(s)\in V^c(p)$. Note that for instance, the condition (i) in the definition of $R^+_1$ is equivalent to the more complex one `$s,t\in\overline{T_1}$ and $sR_1t$', since $sR_1t$ implies that $s\in\overline{T_1}$, and similarly for other conditions. An analogous simplification goes also to the cases (i)-(iii) in the definition $R^+$ in~\cite[Def.~5.9]{Fanetal:2014}.\footnote{In detail, the definition of $R^+$ in~\cite[Def.~5.9]{Fanetal:2014} can be simplified into the following: $sR^+t$ iff one of the following cases holds: (i) $t\in\overline{D}$ and $sR^ct$, (ii) $t=(s,s')\in S^+$, (iii) $s=(t,t')\in S^+$.}

\weg{\begin{proposition} %Let $s,t,u\in S^+$.
If $s\in\overline{T_1}\cap\overline{T_2}$ and $sR^+_1t$, then $h(s)R_1h(t)$.
\end{proposition}

\begin{proof}
Suppose that $sR^+_1t$, to show that $h(s)R_1h(t)$. By supposition and the definition of $R^+_1$, we consider two cases.
\begin{enumerate}
\item[(1)] $t\in\overline{T_1}$ and $sR_1t$.
\end{enumerate}
\end{proof}}

Prop.~\ref{prop.forth-h}---Prop.~\ref{prop.surj-h} together say that $h$ is a surjective $\boxdot$-morphism, and therefore $\M$ is a $\boxdot$-morphic image of $\M^+$.

\begin{proposition}\label{prop.forth-h}[Forth] %Let $s,t,u\in S^+$.
If $sR^+_1t$ and $sR^+_2u$ and $h(t)\neq h(u)$, then $h(s)R_1h(t)$ and $h(s)R_2h(u)$.
\end{proposition}

\begin{proof}
We show a stronger result:

$(\ast)$~~~If $sR^+_1t$ and $sR^+_2u$ and $t\neq u$, then $h(s)R_1h(t)$ and $h(s)R_2h(u)$.

Suppose that $sR^+_1t$ and $sR^+_2u$ and $t\neq u$. Then the arrows from $s$ to $t$ and $u$ are both impossible to be constructed by the condition (iii), since otherwise $s=(t,s',1)$ and $s=(u,s'',2)$, which would entail that $t=u$, contradiction. In the sequel, it suffices to consider the remaining two conditions.

Since $sR^+_1t$, if $t\in \overline{T_1}$ and $sR_1t$, then obviously $s\in\overline{T_1}$, thus $h(s)=s$ and $h(t)=t$, and therefore $h(s)R_1h(t)$; if $t=(s,t',1)\in S^+$, then $sR_1t'$, obviously $s\in\overline{T_1}$, thus $h(s)=s$ and $h(t)=t'$, and therefore $h(s)R_1h(t)$. Similarly, we can show $h(s)R_2h(u)$ by using $sR^+_2u$ instead.
\end{proof}

\begin{proposition}\label{prop.back-h}[Back]
If $h(s)R_1t'$ and $h(s)R_2u'$ and $t'\neq u'$, then there exist $t,u\in S^+$ such that $sR_1^+t, sR_2^+u$ and $h(t)=t'$ and $h(u)=u'$.
\end{proposition}

\begin{proof}
We show a stronger result:

$(\star)~~~$ For any $i\in\{1,2\}$, if $h(s)R_it'$, then there exist $t\in S^+$ such that $sR_i^+t$ and $h(t)=t'$.

%Suppose that $h(s)R_1t'$ and $h(s)R_2u'$ and $t'\neq u'$. It is clear that $h(s)\in\overline{T_1}\cap\overline{T_2}$. Then it must be that $h(s)=s$: otherwise, $h(s)=s'$ for $s=(t,s',i)\in S^+$, where $i\in\{1,2\}$, which would imply that $s'\in\overline{T_1}\cap\overline{T_2}$ and $s'\in T_1\cup T_2$, which is a contradiction.  Hence $sR_1t'$ and $sR_2u'$.
Let $i\in\{1,2\}$. Suppose that $h(s)R_it'$. It is clear that $h(s)\in\overline{T_1}\cap\overline{T_2}$. Then it must be that $h(s)=s$: otherwise, $h(s)=s'$ for $s=(t,s',j)\in S^+$, where $j\in\{1,2\}$, which would imply that $s'\in\overline{T_1}\cap\overline{T_2}$ and $s'\in T_1\cup T_2$, which is a contradiction.  Hence $sR_it'$. Since $t'\in S$, $t'\in \overline{T_i}$ or $t'\in T_i$. If $t'\in \overline{T_i}$, then by the first condition of the definition of $R^+_i$, we infer $sR^+_it'$; by the definition of $h$, $h(t')=t'$. If $t'\in T_i$, then $(s,t',i)\in S^+$, and thus by the second condition of the definition of $R^+_i$, we derive $sR^+_i(s,t',i)$; by the definition of $h$, we get $h((s,t',i))=t'$, as desired.
\end{proof}

%We still need to show that $h$ is surjective.

\begin{proposition}\label{prop.surj-h}
The function $h$ is surjective.
\end{proposition}

\begin{proof}
Suppose that $s\in S$, we need to find a $s'\in S^+$ such that $h(s')=s$.

If $s\in\overline{T_1}$, then $s\in S^+$ and $h(s)=s$; otherwise, $s\in T_1$, then there exists $x\in S$ such that $xR_1s$, thus $(x,s,1)\in S^+$, and then $h((x,s,1))=s$, as desired.
\end{proof}

Now using Prop.~\ref{prop.boxdot-modelproperty}, we immediately have

\begin{lemma}\label{lem.equivalent-b}
For all $s\in S^+$, for all $\phi\in\mathcal{L}(\boxdot)$, we have
$$\M^+,s\vDash\phi\iff \M,h(s)\vDash\phi.$$
\end{lemma}

\weg{\begin{proof}
By induction on $\phi\in\mathcal{L}(\boxdot)$, where the nontrivial case is $\boxdot\phi$.

Suppose that $\M^+,s\nvDash\boxdot\phi$, to show that $\M,h(s)\nvDash\boxdot\phi$. By supposition, there are $t,u\in S^+$ such that $sR^+_1t$ and $sR^+_2u$ and it is {\em not} the case that ($\M^+,t\vDash\phi\iff \M^+,u\vDash\phi$). By induction hypothesis, it is {\em not} the case that ($\M,h(t)\vDash\phi\iff \M,h(u)\vDash\phi$), and thus $h(t)\neq h(u)$. By Prop.~\ref{prop.forth-h}, we derive that $h(s)R_1h(t)$ and $h(s)R_2h(u)$. Therefore, $\M,h(s)\nvDash\boxdot\phi$.

Conversely, assume that $\M,h(s)\nvDash\boxdot\phi$, to demonstrate that $\M^+,s\nvDash\boxdot\phi$. By assumption, there exist $t',u'\in S$ such that $h(s)R_1t'$ and $h(s)R_2u'$ and it is {\em not} the case that ($\M,t'\vDash\phi\iff \M,u'\vDash\phi$). Clearly, $t'\neq u'$. Now by Prop.~\ref{prop.back-h}, there are $t,u\in S^+$ such that $sR^+_1t$ and $sR^+_2u$ and $h(t)=t'$ and $h(u)=u'$. By induction hypothesis, we can show that it is {\em not} the case that ($\M^+,t\vDash\phi\iff \M^+,u\vDash\phi$). Therefore, $\M^+,s\nvDash\boxdot\phi$.
\end{proof}}

To finish the completeness of ${\bf KB^\boxdot}$, we need also show that $\M^+$ is symmetric.
\begin{lemma}\label{lem.symmetric}
$\M^+$ is symmetric.
\end{lemma}

\begin{proof}
We need to show that $R_1^+$ and $R^+_2$ are both symmetric. We show only the symmetry of $R^+_1$, since the symmetry of $R^+_2$ can be proved analogously.

Suppose for any $s,t\in S^+$ we have $sR^+_1t$, to show that $tR^+_1s$. According to the definition of $R^+$, we consider three conditions.
\begin{itemize}
\item $t\in\overline{T_1}$ and $sR_1t$. Then $tR_ju$ for some $u\in S$, where $j\in \{1,2\}$. Since $\M$ is quasi-symmetric, we have $tR_1s$. Obviously, $s\in \overline{T_1}$. It then follows that $tR^+_1s$.
\item $t=(s,t',1)\in S^+$ for some $t'$. By the third condition of the definition of $R^+_1$, it follows that $tR^+_1s$.
\item $s=(t,s',1)\in S^+$ for some $s'$. By the second condition of the definition of $R^+_1$, it follows that $tR^+_1s$.
\end{itemize}
\end{proof}

\begin{theorem}
${\bf KB^\boxdot}$ is strongly complete with respect to the class of symmetric bimodal frames.
\end{theorem}

\begin{proof}
Let $\Sigma$ be a consistent set. By Thm.~\ref{thm.comp-quasi-symmetric}, $\Sigma$ is satisfiable in a quasi-symmetric bimodal model, say $(\M,s)$. Construct $\M^+$ from $\M$ as in Def.~\ref{def.symmetric-model}. As $h$ is surjective (Prop.~\ref{prop.surj-h}), there exists $x\in S^+$ such that $h(x)=s$, thus $\M,h(x)\vDash\Sigma$. By Lemma~\ref{lem.equivalent-b} and Lemma~\ref{lem.symmetric}, $\Sigma$ is satisfiable in a symmetric bimodal model $\M^+$, as required.
\end{proof}

%Symmetric case will be hard? Note that $\phi\to\boxdot((\boxdot\phi\land\boxdot(\phi\to\psi)\land\neg\boxdot\psi)\to\chi)$ is valid over symmetric frames.

\subsubsection{Transitive-like and Euclidean-like logics}

In contingency logic, $\Delta\phi\to\Delta(\Delta\phi\vee\psi)$ and $\neg\Delta\phi\to\Delta(\neg\Delta\phi\vee\psi)$, are added in the minimal contingency logic to axiomatize the class of transitive frames and the class of Euclidean frames, respectively, see e.g.~\cite{Fanetal:2015}. It is then quite natural to expect that their $\boxdot$-counterparts $\boxdot\phi\to\boxdot(\boxdot\phi\vee\psi)$ (denoted $\boxdot4$) and $\neg\boxdot\phi\to\boxdot(\neg\boxdot\phi\vee\psi)$, can be used to axiomatize this generalized logic over the same classes. Unfortunately, it turns out to be wrong, since $\boxdot4$ and $\boxdot5$ are not sound. In what follows, instead of showing this directly, we show that one of the weaker versions of each of them, viz. $\boxdot \phi\to\boxdot\boxdot\phi$ (denoted w$\boxdot4$) and $\neg\boxdot \phi\to\boxdot\neg\boxdot\phi$ (denoted w$\boxdot5$), are invalid over the corresponding frame class.

\begin{proposition}
w$\boxdot4$ is invalid over the class of transitive bimodal frames.
\end{proposition}

\begin{proof}
Consider the following model $\M=\lr{S,R_1,R_2,V}$:
\weg
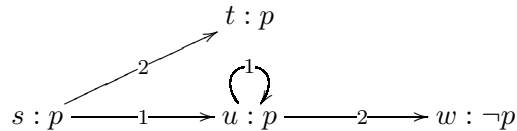
{\[
\xymatrix{&&t: p&&v:\neg p\\
s:p\ar[rr]|1\ar[urr]|2\ar@/_20pt/[rrrr] |1&&u: p\ar[urr]|2\ar[rr]|1&&w: p\\}
\]}
\[
\xymatrix{&&t: p&&\\
s:p\ar[rr]|1\ar[urr]|2&&u: p\ar@(ul,ur)|1\ar[rr]|2&&w:\neg p\\}
\]
It can be checked easily that both $R_1$ and $R_2$ are transitive, and thus $\M$ is transitive. On one hand, since all $R_1$-successors $u$ and $R_2$-successors $t$ of $s$ agree on the truth value of $p$, we have $s\vDash\boxdot p$. On the other hand, because some $R_1$-successor $u$ and some $R_2$-successor $w$ of $u$ do not agree on the truth value of $p$, we obtain $u\nvDash\boxdot p$; since $t$ has no any successors, $t\vDash\boxdot p$, and thus $s\nvDash\boxdot\boxdot p$. Therefore, $s\nvDash\boxdot p\to\boxdot\boxdot p$.
%On one hand, since all $R_1$-successors $u$ and $R_2$-successors $t$ of $s$ agree on the truth value of $p$, we have $s\vDash\boxdot p$. On the other hand, because some $R_1$-successor $v$ and some $R_2$-successor $w$ of $u$ do not agree on the truth value of $p$, we obtain $u\nvDash\boxdot p$; since $t$ has no any successors, $t\vDash\boxdot p$, and thus $s\nvDash\boxdot\boxdot p$. Therefore, $s\nvDash\boxdot p\to\boxdot\boxdot p$.
\end{proof}

%Also, the $\boxdot$-counterpart (denoted $\boxdot5$), $\neg\boxdot\phi\to\boxdot(\neg\boxdot\phi\vee\psi)$, is unsound over the class of Euclidean bimodal frames. We show the invalidity of one of its weaker versions, i.e. $\neg\boxdot\phi\to\boxdot\neg\boxdot\phi$ (denoted w$\boxdot5$) instead.
\begin{proposition}
w$\boxdot5$ is invalid over the class of Euclidean bimodal frames.
\end{proposition}

\begin{proof}
Consider the following Euclidean model $\M'=\lr{S,R_1,R_2,V}$:

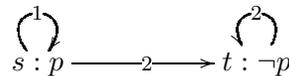
\weg{\[
\xymatrix{s:p\ar[rr]|1\ar[drr]|2&&t:p\ar@(ur,ul)|1\\
&&u:\neg p\ar@(dr,dl)|2\ar[u]|1\\}
\]}

\[
\xymatrix{s: p\ar@(ul,ur)|1\ar[rr]|2&&t:\neg p\ar@(ur,ul)|2\\}
\]

On one hand, $s\vDash\neg\boxdot p$: because $sR_1s$ and $sR_2t$ and $s\vDash p$ but $t\nvDash p$. On the other hand, $s\nvDash\boxdot\neg\boxdot p$: as $t$ has only a single successor, $t\vDash\boxdot p$, i.e. $t\nvDash\neg\boxdot p$, and thus $s\nvDash\boxdot\neg\boxdot p$. Therefore, $s\nvDash\neg\boxdot p\to\boxdot\neg\boxdot p$.
\end{proof}

Denote ${\bf K4^\boxdot}={\bf K^\boxdot}+\boxdot4$ and ${\bf K5^\boxdot}={\bf K^\boxdot}+\boxdot5$. As we have seen, ${\bf K4^\boxdot}$ and ${\bf K5^\boxdot}$ are not the transitive $\boxdot$-logic and Euclidean $\boxdot$-logic, respectively. It is then natural to ask which logics both proof systems are; in other words, which classes of frames are characterized by ${\bf K4^\boxdot}$ and ${\bf K5^\boxdot}$, respectively.

\weg{Consider the following frame properties: where $i,j\in\{1,2\}$,
\[
\begin{array}{|lll|}
\hline
\text{Names}&\text{FO-correspondents}&\text{Abbreviations}\\
\hline
\text{Quasi-transitivity}&\forall xyz (xR_iy\land yR_jz\to xR_jz)&qt\\
\text{Pseudo-transitivity}&\forall xyz (xR_iy\land yR_jz\to xR_1z\land xR_2z)&pt\\
\text{Quasi-Euclidicity}&\forall xyz(xR_iy\land xR_jz\to yR_jz)&qe\\
\text{Pseudo-Euclidicity}&\forall xyz(xR_iy\land xR_jz\to yR_1z\land yR_2z)&pe\\
\hline
\end{array}
\]}

\weg{Quasi-transitivity: $\forall xyz (xR_iy\land yR_jz\to xR_jz)$, $qt$ for short,

%Pseudo-transitivity: $\forall xyz (xR_iy\land yR_jz\to xR_kz)$
Pseudo-transitivity: $\forall xyz (xR_iy\land yR_jz\to xR_1z\land xR_2z)$, $pt$ for short,

Quasi-Euclidicity: $\forall xyz(xR_iy\land xR_jz\to yR_jz)$, $qe$ for short,

%Pseudo-Euclidicity: $\forall xyz(xR_iy\land xR_jz\to yR_kz)$.
Pseudo-Euclidicity: $\forall xyz(xR_iy\land xR_jz\to yR_1z\land yR_2z)$, $pe$ for short.}

We remind the reader of the properties $qt$, $pt$, $qe$, $pe$ at the beginning of Sec.~\ref{sec.extensions}. It is not hard to see that $pt$ is stronger than $qt$, and $pe$ is stronger than $qe$, thus every $pt$-frame/model is a $qt$-frame/model, and every $pe$-frame/model is a $qe$-frame/model. We use $\Gamma\vDash_{qt}\phi$ to mean that $\phi$ is a semantical consequence of $\Gamma$ over the class of $qt$-frames, that is, for every $qt$-model $\M$ and every state $s$ in $\M$, if $\M,s\vDash\psi$ for all $\psi\in\Gamma$, then $\M,s\vDash\phi$.  Similar meanings goes to $\Gamma\vDash_{pt}\phi$, $\Gamma\vDash_{qe}\phi$, and $\Gamma\vDash_{pe}\phi$. We will show that ${\bf K4^\boxdot}$ is sound and strongly complete with respect to both the class of $qt$-frames and the class of $pt$-frames, and ${\bf K5^\boxdot}$ is sound and strongly complete with respect to both the class of $qe$-frames and the class of $pe$-frames.

Before showing the soundness and strong completeness of ${\bf K4^\boxdot}$ and ${\bf K5^\boxdot}$, it is worth remarking that w$\boxdot4$ and w$\boxdot5$ are provable in ${\bf K4^\boxdot}$ and ${\bf K5^\boxdot}$, respectively, by letting $\psi$ in $\boxdot4$ and $\boxdot5$ be $\bot$.

To simplify the proofs below, we provide two useful results.

\begin{proposition}\label{prop.tran}
Define $\M^c$ w.r.t. ${\bf K4^\boxdot}$ as in Def.~\ref{def.cm-boxdot} and $sR^c_it$ for $i\in\{1,2\}$. If $\boxdot\phi\in s$, then $\boxdot\phi\in t$.
\end{proposition}

\begin{proof}
Suppose that $sR^c_it$ for $i\in\{1,2\}$ and $\boxdot\phi\in s$. Then there exists $\chi$ such that $\neg\boxdot\chi\in s$ and $(\ast)$ for all $\phi$, if $\boxdot\phi\land\boxdot(\chi\to\phi)\in s$, then $\phi\in t$.

since $\boxdot\phi\in s$, by w$\boxdot4$, we have $\boxdot\boxdot\phi\in s$; by $\boxdot4$, we obtain that $\boxdot(\boxdot\phi\vee\neg\chi)\in s$, that is, $\boxdot(\chi\to\boxdot\phi)\in s$, then by $(\ast)$, it follows that $\boxdot\phi\in t$.
\end{proof}

\begin{proposition}\label{prop.eucl}
Define $\M^c$ w.r.t. ${\bf K5^\boxdot}$ as in Def.~\ref{def.cm-boxdot} and $sR^c_it$ for $i\in\{1,2\}$. If $\boxdot\phi\in t$, then $\boxdot\phi\in s$.
\end{proposition}

\begin{proof}
Suppose that $sR^c_it$ for $i\in\{1,2\}$ and $\neg\boxdot\phi\in s$. Then there exists $\chi$ such that $\neg\boxdot\chi\in s$ and $(\star)$ for all $\psi$, if $\boxdot\psi\land\boxdot(\chi\to\psi)\in s$, then $\phi\in t$.

since $\neg\boxdot\phi\in s$, by w$\boxdot5$ it follows that $\boxdot\neg\boxdot\phi\in s$; by $\boxdot5$, it follows that $\boxdot(\neg\boxdot\phi\vee\neg\chi)\in s$, i.e. $\boxdot(\chi\to\neg\boxdot\phi)\in s$. Then using $(\star)$, we derive that $\neg\boxdot\phi\in t$.
\end{proof}

We are now ready to show the soundness and strong completeness of ${\bf K4^\boxdot}$ and ${\bf K5^\boxdot}$.

\begin{theorem}
Let $\phi\in\mathcal{L}(\boxdot)$. The following conditions are equivalent:
\begin{itemize}
\item[(a)] $\Gamma\vdash_{{\bf K4^\boxdot}}\phi$
\item[(b)] $\Gamma\vDash_{qt}\phi$
\item[(c)] $\Gamma\vDash_{pt}\phi$.
\end{itemize}
\end{theorem}

\begin{proof} We show $(a)\Rightarrow (b)\Rightarrow (c)\Rightarrow (a)$.

$(a)\Rightarrow (b)$: By soundness of ${\bf K^\boxdot}$, it suffices to show that $\boxdot4$ is valid on the class of $qt$-frames.

If not, there exists a $qt$-model $\M=\lr{S,R_1,R_2,V}$ and a state $s\in S$ such that $\M,s\vDash\boxdot\phi$ but $s\nvDash\boxdot(\boxdot\phi\vee\psi)$. Then for some $t$ and $u$, it holds that $sR_1t$ and $sR_2u$ and $t\vDash \boxdot\phi\vee\psi\not\iff u\vDash\boxdot\phi\vee\psi$. W.l.o.g. we assume that $t\nvDash\boxdot\phi\vee\psi$ and $u\vDash\boxdot\phi\vee\psi$. From $t\nvDash\boxdot\phi\vee\psi$ it follows that $t\nvDash\boxdot\phi$, and thus there are $v,w$ such that $tR_1v$ and $tR_2w$ and $(v\vDash\phi\not\iff w\vDash\phi)$. By $sR_1t$, $tR_1v$, $tR_2w$ and the property $(qt)$ of $\M$, we have $sR_1v$ and $sR_2w$, which together with the fact that $s\vDash\boxdot\phi$ implies that $(v\vDash\phi\iff w\vDash\phi)$: a contradiction.

$(b)\Rightarrow (c)$: this is because every $pt$-model is a $qt$-model.

$(c)\Rightarrow (a)$: Define $\M^c$ w.r.t. ${\bf K4^\boxdot}$ as in Def.~\ref{def.cm-boxdot}. It is sufficient to show that $\M^c$ is a $pt$-model. %Since $R^c_1=R^c_2$, this means that it is enough to prove that $\forall s,t,u\in S^c(sR^c_1t\land tR^c_1u\to sR^c_1u)$.

Suppose for $i,j\in\{1,2\}$ that $sR^c_it$ and $tR_j^cu$. Then there exists $\chi$ such that $\neg\boxdot\chi\in s$ and (a) for all $\phi$, if $\boxdot\phi\land\boxdot(\chi\to\phi)\in s$, then $\phi\in t$, and there is a $\psi$ such that $\neg\boxdot\psi\in t$ and (b) for all $\phi$, if $\boxdot\phi\land\boxdot(\psi\to\phi)\in t$, then $\phi\in u$. To show $sR_1^cu$ and $sR^c_2u$, it suffices to demonstrate that for all $\phi$, if $\boxdot\phi\land\boxdot(\chi\to\phi)\in s$, then $\phi\in u$. For this, let $\phi$ be arbitrary such that $\boxdot\phi\land\boxdot(\chi\to\phi)\in s$. In what follows, we will show that $\boxdot\phi\land\boxdot(\psi\to\phi)\in t$, which by (b) implies that $\phi\in u$.
\begin{itemize}
\item $\boxdot\phi\in t$: direct by $sR^c_it$ and $\boxdot\phi\in s$ and Prop.~\ref{prop.tran}.%since $\boxdot\phi\in s$, by w$\boxdot4$, we have $\boxdot\boxdot\phi\in s$; by $\boxdot4$, we obtain that $\boxdot(\boxdot\phi\vee\neg\chi)\in s$, that is, $\boxdot(\chi\to\boxdot\phi)\in s$, then by (a), it follows that $\boxdot\phi\in t$.
\item $\boxdot(\psi\to\phi)\in t$: from $\boxdot(\chi\to\phi)\in s$ (i.e. $\boxdot(\neg\phi\to\neg\chi)\in s$) and $\neg\boxdot\chi\in s$ (i.e. $\boxdot\neg\chi\notin s$), it follows by axiom $\boxdot\text{CON}$ that $\boxdot(\phi\to\neg\chi)\notin s$, namely $\boxdot(\neg\phi\vee\neg\chi)\notin s$. Thanks to $\boxdot\phi\in s$, by axiom $\boxdot\text{DIS}$ we infer that $\boxdot(\phi\vee\neg\psi)\in s$, that is, $\boxdot(\psi\to\phi)\in s$. Then by Prop.~\ref{prop.tran} again, we conclude that $\boxdot(\psi\to\phi)\in t$. %similar to the proof of $\boxdot\phi\in t$, we can finally conclude that $\boxdot(\psi\to\phi)\in t$.
\end{itemize}
\end{proof}

\begin{theorem}
Let $\phi\in\mathcal{L}(\boxdot)$. The following conditions are equivalent:
\begin{itemize}
\item[(a)] $\Gamma\vdash_{{\bf K5^\boxdot}}\phi$
\item[(b)] $\Gamma\vDash_{qe}\phi$
\item[(c)] $\Gamma\vDash_{pe}\phi$.
\end{itemize}
\end{theorem}

\begin{proof}
We show $(a)\Rightarrow (b)\Rightarrow (c)\Rightarrow (a)$.

$(a)\Rightarrow (b)$: by soundness of ${\bf K^\boxdot}$, it is sufficient to show that $\boxdot5$ is valid on the class of $qe$-frames.

If not, there exists $qe$-model $\M=\lr{S,R_1,R_2,V}$ and state $s\in S$ such that $\M,s\vDash\neg\boxdot\phi$ but $s\nvDash\boxdot(\neg\boxdot\phi\vee\psi)$. From $s\vDash\neg\boxdot\phi$, it follows that for some $t,u$ such that $sR_1t$ and $sR_2u$ and $t\vDash\phi\not\iff u\vDash\phi$. From $s\nvDash\boxdot(\neg\boxdot\phi\vee\psi)$, it follows that for some $v,w$ such that $sR_1v$ and $sR_2w$ and $v\vDash\neg\boxdot\phi\vee\psi\not\iff w\vDash\neg\boxdot\phi\vee\psi$. W.l.o.g. we assume that $v\vDash\neg\boxdot\phi\vee\psi$ and $w\vDash\boxdot\phi\land\neg\psi$. By $sR_2w$ and $sR_1t$ and $sR_2u$ and the property $(qe)$ of $\M$, we infer $wR_1t$ and $wR_2u$. Due to $w\vDash\boxdot\phi$, we have $t\vDash\phi\iff u\vDash\phi$: a contradiction.

$(b)\Rightarrow (c)$: This is due to the fact that every $pe$-model is a $qe$-model.

$(c)\Rightarrow (a)$: Define $\M^c$ w.r.t. ${\bf K5^\boxdot}$ as in Def.~\ref{def.cm-boxdot}. The remainder is to prove that $\M^c$ is a $pe$-model.

Suppose for $i,j\in\{1,2\}$ that $sR^c_it$ and $sR^c_ju$. Then there exists $\chi$ such that $\neg\boxdot\chi\in s$ and $(\dag)$ for all $\phi$, if $\boxdot\phi\land\boxdot(\chi\to\phi)\in s$, then $\phi\in t$, and there is a $\psi$ such that $\neg\boxdot\psi\in s$ and $(\dag\dag)$ for all $\phi$, if $\boxdot\phi\land\boxdot(\psi\to\phi)\in s$, then $\phi\in u$. To show $tR^c_1u$ and $tR^c_2u$, we need to find a $\delta$ such that $\neg\boxdot\delta\in t$ and for all $\phi$, if $\boxdot\phi\land\boxdot(\delta\to\phi)\in t$, then $\phi\in u$. We show that $\chi$ is a desired $\delta$.
\begin{itemize}
\item $\neg\boxdot\chi\in t$: otherwise, by Prop.~\ref{prop.eucl}, we would derive $\boxdot\chi\in s$: a contradiction.%since $\neg\boxdot\chi\in s$, by w$\boxdot5$ it follows that $\boxdot\neg\boxdot\chi\in s$; by $\boxdot5$, it follows that $\boxdot(\neg\boxdot\chi\vee\neg\chi)\in s$, i.e. $\boxdot(\chi\to\neg\boxdot\chi)\in s$. Then using $(\dag)$, we derive that $\neg\boxdot\chi\in t$.
\item Assume for any $\phi$ such that $\boxdot\phi\land\boxdot(\chi\to\phi)\in t$, we only need show that $\phi\in u$. By assumption and Prop.~\ref{prop.eucl}, $\boxdot\phi\land\boxdot(\chi\to\phi)\in s$. As $\neg\boxdot\chi\in s$, $\boxdot\neg\chi\notin s$; as $\boxdot(\chi\to\phi)\in s$, $\boxdot(\neg\phi\to\neg\chi)\in s$. Thus by axiom $\boxdot\text{CON}$, it follows that $\boxdot(\phi\to\neg\chi)\notin s$, viz. $\boxdot(\neg\phi\vee\neg\chi)\notin s$. From this and $\boxdot\phi\in s$ and axiom $\boxdot\text{DIS}$, we have $\boxdot(\phi\vee\neg\psi)\in s$, that is, $\boxdot(\psi\to\phi)\in s$. Now applying $(\dag\dag)$, we get $\phi\in u$, as desired.
\end{itemize}
\end{proof}

\weg{
\section{Relativized}

\[
\begin{array}{lll}
\M,s\vDash\boxdot^\psi\phi&\iff &\text{for all }t,u,\text{ if }sR_1t\text{ and }sR_2u\text{ and }\M,t\vDash\psi\text{ and }\M,u\vDash\psi,\\
&&\text{then }(\M,t\vDash\phi\iff \M,u\vDash\phi).\\
\end{array}
\]

$\boxdot^\psi\phi\land \boxdot\psi\to\boxdot(\psi\to\phi)$.

$\boxdot\phi\to\boxdot^\psi\phi$.

Question: is the relativized operator definable in terms of $\boxdot$?

\section{}

\[
\begin{array}{lll}
\M,s\vDash D(\psi,\phi)&\iff &\text{for all }t,u,\text{ if }sR_1t\text{ and }sR_2u,\\
&&\text{if }(\M,t\vDash\psi\iff\M,u\vDash\psi),\\
&&\text{then }(\M,t\vDash\phi\iff \M,u\vDash\phi).\\
\end{array}
\]

\section{Adding public announcements}

$[\psi]\boxdot\phi\lra(\psi\to\boxdot[\psi]\phi)$?

Conjecture: the addition of public announcements will increase the expressivity of $\mathcal{L}(\boxdot)$.}

\section{Conclusion and Future work}\label{sec.conclusion}

In this paper, we proposed the operator $\boxdot$ for the generalized noncontingency and the operator $\boxplus$ for pseudo noncontingency, which are obtained by slightly adapting two equivalent semantics of noncontingency operator. We showed that $\mathcal{L}(\boxdot)$ is less expressive than $\mathcal{L}(\boxplus)$ over five basic model classes. Besides, the two logics cannot define the five basic frame properties, with the aid of a notion of $\boxdot$-morphisms. We then presented the minimal logic of $\mathcal{L}(\boxplus)$, which also characterizes the class of serial bimodal frames. Moreover, we axiomatized $\mathcal{L}(\boxdot)$ over various frame classes, among which the completeness of serial logic and of symmetric logic were shown via the notion of $\boxdot$-morphisms.

There are a lot of future work to be continued. For instance, the axiomatizations of $\mathcal{L}(\boxplus)$ over the class of frames with other special properties, including reflexivity, transitivity, symmetry, Euclidicity; the axiomatizations of $\mathcal{L}(\boxdot)$ over the class of transitive frames and over the class of Euclidean frames.

\section*{Acknowledgements}

This research is financially supported by the project 17CZX053 of National Social Science Fundation of China.

\bibliographystyle{alpha}
\bibliography{biblio2017}
\end{document}